\documentclass{amsart}

\usepackage[utf8]{inputenc}
\usepackage[english]{babel}
\usepackage[IL2]{fontenc}

\usepackage{standalone}
\usepackage{tikz}
\usetikzlibrary{shapes}

\usepackage{amsaddr}

\usepackage{amssymb}
\usepackage{amsmath}
\usepackage{amsthm}
\usepackage{mathtools}
\usepackage{graphicx}
\usepackage{dynkin-diagrams}
\usepackage{enumitem}
\usepackage{array}
\usepackage{multirow}
\usepackage{booktabs}
\usepackage{xcolor}

\usepackage{algpseudocode}

\usepackage{hyperref}

\theoremstyle{plain}

\newtheorem{theorem}{Theorem}

\newtheorem{lemma}[theorem]{Lemma}
\newtheorem{proposition}[theorem]{Proposition}

\newtheorem{algorithm}[theorem]{Algorithm}

\theoremstyle{definition}

\newtheorem*{definition}{Definition}
\newtheorem*{example}{Example}
\newtheorem*{observation}{Observation}
\newtheorem*{convention}{Convention}

\theoremstyle{remark}

\newtheorem*{remark}{Remark}

\newcommand{\N}{\mathbb{N}}
\newcommand{\R}{\mathbb{R}}

\newcommand{\xlra}[1]{\xrightarrow{\ #1\ }}
\newcommand{\mc}[1]{\mathcal{#1}}
\newcommand{\mfk}[1]{\mathfrak{#1}}

\newcommand{\on}[1]{\operatorname{#1}}

\newcommand{\di}{\mathrm{d}}

\begin{document}

\title{The harmonic curvature of 3-link snake robots} 
\author{Martin Dole\v{z}al}
\address{Department of Mathematics and Statistics, Masaryk University, Brno, Czech Republic, and Department of Mathematics, University of Hradec Kr\'{a}lov\'{e}, Hradec Kr\'{a}lov\'{e}, Czech Republic}

\begin{abstract}
    The $3$-link snake robot is an example of a non-holonomic mechanical system with rank $2$ distribution in a $5$-dimensional configuration space. It is one of $(2,3,5)$-geometries, and as such, it admits a description by a parabolic geometry of type $(G_2,P)$. Another example of $(2,3,5)$-geometry was well-studied years ago, and it is known that for balls rolling one over the other without slipping or twisting, if the ratio of ball radii is $1:3$, then it is locally isomorphic to the flat model in sense of $(G_2,P)$ parabolic geometries. Answering a question by P.~Nurowski, we are looking for parameters of the $3$-link snake robots yielding a locally flat $(2,3,5)$-geometry. We extend the observation of a previous paper that the distributions of the snake robots contain bases generating finite dimensional Lie algebras. We exploit this observation to simplify the exterior calculus of the robots' geometry. This allows us to implement an effective normalization procedure and we obtain an explicit binary quartic invariant of the robot. Finally, we show that it does not vanish for any of the parameters. Therefore, the locally flat model cannot be achieved for these types of snake robots.
\end{abstract}



\thanks{I would like to thank my supervisor Jan Slov\'{a}k for many helpful comments during writing this paper, and to Jan Gregorovi\v{c} for helping me with an early version of the normalization algorithm implementation.\\ \hspace*{1.5em}
This work was supported by Masaryk University under Grant MUNI/A/1648/2025 for doctoral students and further supported from COST Action CaLISTA CA21109 supported by COST (European Cooperation in Science and Technology, \url{www.cost.eu}), and Horizon Europe Framework Programme (HORIZON), project nr. 101086123, CaLIGOLA.
}

\maketitle

\section{Introduction}

We are going to study kinematics of a planar robot, so called $3$-link snake, from the geometric viewpoint.

The $3$-link snake robot is a type of planar robot consisting of three segments connected by two joints, whose movements are restricted by a non-slipping condition on its three fixed wheels that are placed in the middle of each segment. Any position of the robot can be uniquely described as a point in the configuration space
\[ \mc{M} = S^1 \times S^1 \times S^1 \times \R^2,\]
i.e.~by three angles and a point on the plane. The non-slipping condition left two degrees of freedom, which are usually realized by setting the angle on both of the joints. These two elemental movements infinitesimally form two vector fields that span the $2$-dimensional subspace $\mc{D}_x \subset T_x\mc{M}$ of the tangent space in any generic point of $\mc{M}$%
\footnote{The excluded points of $\mc{M}$ are the position of straight or folded snake, and the positions in which the wheels of the snake lie on a circle.},
i.e.~they generate the rank $2$ distribution $\mc{D}$ corresponding to the robot kinematics. This distribution is bracket-generating with the growth vector $(2,3,5)$, i.e.~one can obtain $5$ independent vector fields in two steps by applying the Lie bracket to vector fields in the distribution $\mc{D}$. By Chow-Rashewskii theorem, any bracket-generating distribution is controllable, therefore the robot can move to any other position on an open subset where $\mc{D}$ is bracket-generating.

A study of this robot's possible movements probably first appeared in \cite{pptm1999}, although similar snake-like robots were studied earlier by S.~Hirose \cite{sh1993}. In last two decades, several authors studied its point-to-point controllability using modern geometric tools, e.g.~\cite{mi2009} and \cite{jhanpv2016}. From the viewpoint of Cartan geometries, the $3$-link snake was studied by Nurowski, see his presentation \cite{pn2014} and recent publication \cite{tf2024}.

The configuration space $\mc{M}$ enjoys the structure of $(2,3,5)$ filtered manifold and uniquely corresponds to a parabolic geometry of type $(G_2,P)$ with a regular normal Cartan connection. Such structures were studied already by Cartan in \cite{ec1910}, who described most cases that admit a transitive Lie group action that preserves $\mc{D}$. Nurowski considered the $3$-link snake in a more general way, and observed that its kinematics forms a $(2,3,5)$ filtered manifold for an arbitrary choice of segments' lengths and wheels positions.

Another $(2,3,5)$-geometry is one describing motion of two balls, one rolling on the other without slipping or spinning. It is known that for the ratio of the balls $1:3$, the geometry becomes the locally flat model of the $(G_2,P)$ parabolic geometry, i.e.~with $14$-dimensional group of symmetries, while for all other parameters it forms a homogeneous model with the symmetry algebra $\mfk{so}(3)\times \mfk{so}(3)$, see \cite[p.~16]{dt2022} and \cite{gb2009}. Naturally, we should like to know whether there are some locally flat $3$-link snake robot for specific choices of parameters. This would be interesting not only from a theoretical point of view, but also very helpful for the control theory problem. Nurowski posed this question in \cite{pn2014}. Although he didn't find the answer in full generality due to the computational complexity, he answered negatively the question for the special position of the central wheel with use of CR-geometries.

In this paper, we are answering this question negatively. We explicitly compute an invariant called Cartan binary quartic form for all generalized $3$-link snakes considered by Nurowski. This bi-quartic form is an invariant of $(2,3,5)$-geometries, already found by Cartan in \cite{ec1910}, equivalent to the notion of the harmonic curvature in parabolic geometry. We use the parabolic geometry approach to find the corresponding regular, normal Cartan connection for all possible segment lengths and wheel positions, and the harmonic curvature will immediately follow.

Following \cite{md2024}, we consider vector fields that correspond to the steering of the central wheel, i.e.~the rotation and the forward motion, although the distribution $\mc{D}$ is usually generated by the vector fields corresponding to the rotations of the joints. The resulting distributions are the same in any generic position, and one can easily express the first vector fields by the others. Also in \cite{md2024}, we observed that these two vector fields even generate a $6$-dimensional Lie subalgebra of the Lie algebra of all vector fields $\mfk{X}(\mc{M})$. Although, there is no obvious way how to find such vector fields in the distributions $\mc{D}$ generating finite dimensional Lie subalgebra of $\mfk{X}(\mc{M})$ in general, we are able to do so for all general snakes here. Such Lie subalgebras offer a way how to substitute into the structure equations of an coframe adapted to the $(2,3,5)$-filtration, which makes it handy for the normalization procedure leading to the regular, normal Cartan connection.

We proceed in the following order. For the convenience of the reader, we briefly recall the result of the author's last paper \cite{md2024} and present the techniques of this paper on the simplest example of the snake robot. Then, we show vector fields of the distribution of the generalized $3$-link snake for any choice of parameters, and how they form finite dimensional Lie subalgebra of $\mfk{X}(\mc{M})$. In section \ref{section:coframe_structure}, we derive the handy form of structure equations. Afterwards, we explain the normalization procedure for parabolic geometries associated with $(2,3,5)$-distributions, following Section 3.1 of \cite{parabook} and \cite{dt2022}. Finally, we present the results and remarks.

Most of the computations were implemented in \texttt{Maple17} with  \texttt{DifferentialGeometry} package (\cite{dgpkg}). Some computations were made in \texttt{Python} with \texttt{SymPy} in \texttt{Jupyter} notebook. The worksheets for \texttt{Maple17} and the notebook for \texttt{Python} have been attached to this paper, see the \emph{attached files} \cite{attached_files}. References to the files are provided wherever in the text claims are supported by computer algebra computations.

\section{Standard 3-link snake} \label{section:std_snake}

The $3$-link snake is a planar robot that consists of three segments of length~$1$ equipped with wheels, one in the center of each segment, and the segments are connected by two movable joints. Any position of the $3$-link snake uniquely corresponds to a point of the $5$-dimensional configuration space
\[ \mc{M} = S^1 \times S^1 \times S^1 \times \R^2.\]
Its movements are restricted by non-slipping conditions in the wheels, i.e.~the velocity vector of each segment's center is collinear with the direction of the segment. All its allowed infinitesimal movements form rank $2$ distribution $\mc{D}$ on the configuration space $\mc{M}$.

\begin{figure}[hbt]
  \centering
  \includegraphics[width=0.6\textwidth]{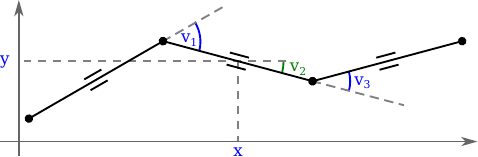}
  \caption{The coordinate system $(v_1,v_2,v_3,x,y)$ on $\mc{M}$}
  \label{fig:snake_std}
\end{figure}

In order to describe the distribution $\mc{D}$, we choose a coordinates system $(v_1, v_2, v_3, x, y)$ of $\mc{M}$ that is chosen as shown in Figure \ref{fig:snake_std}. We want the vector fields in the distribution $\mc{D}$ to satisfy the non-slipping condition in the wheels.

\begin{proposition} \label{prop:generators_standard}
    The distribution $\mc{D}$ of admissible velocities of the robot admits the following vector fields as a basis of $\mc{D}$ in each point of $\mc{M}$.
    \begin{equation}\label{generators_standard}
    \begin{gathered}
        X_1 = (1 + \cos{v_1})\, \partial_{v_1} - \partial_{v_2} + (1 + \cos{v_3})\, \partial_{v_3},\\
        X_2 = -2 \sin v_1\, \partial_{v_1} + 2 \sin v_3\, \partial_{v_3} + \cos v_2\, \partial_{x} + \sin v_2\, \partial_{y}.
    \end{gathered}
    \end{equation}
\end{proposition}
\begin{proof}
    As seen from the Figure \ref{fig:snake_nonslipping}, the velocity vectors of the segment centers, all considered in $\R^2$, must satisfy
    {\footnotesize
    \begin{align*}
        \tfrac{\di}{\di t} (x,y) &\perp (-\sin v_2, \cos v_2),\\
        \tfrac{\di}{\di t} \bigl( (x,y) + \tfrac{1}{2}(\cos v_2, \sin v_2) + \tfrac{1}{2}(\cos(v_3 + v_2), \sin (v_3 + v_2)) \bigr) &\perp (-\sin (v_3+v_2), \cos (v_3 + v_2)),\\
        \tfrac{\di}{\di t} \bigl( (x,y) - \tfrac{1}{2}(\cos v_2, \sin v_2) - \tfrac{1}{2}(\cos(v_1 + v_2), \sin (v_1 + v_2)) \bigr) &\perp (-\sin (v_1+v_2), \cos (v_1 + v_2)).
    \end{align*}}%
    We rewrite them using the standard scalar product, and they correspond to vanishing of the three linearly independent one-forms
    \begin{align*}
        \varphi_1 &= -\sin v_2\, \di x + \cos v_2\, \di y,\\
        \varphi_2 &= -\sin(v_3 + v_2)\, \di x + \cos(v_3 + v_2)\, \di y + \tfrac{1+\cos v_3}{2}\, \di v_2 + \tfrac{1}{2}\, \di v_3,\\
        \varphi_3 &= -\sin(v_1 + v_2)\, \di x + \cos(v_1 + v_2)\, \di y - \tfrac{1}{2}\, \di v_1 - \tfrac{1+\cos v_1}{2}\, \di v_2.
    \end{align*}
    These one-forms are annihilated by the vector fields~\eqref{generators_standard}, as one can show by direct computation, e.g.~in Maple \cite[\texttt{std\_snake.mw}]{attached_files}.
    The vector fields~\eqref{generators_standard} are obviously linearly independent in any point of $\mc{M}$, as one contains $\partial_{v_2}$ term and the other contains $\partial_x, \partial_y$, therefore they generate the distribution $\mc{D}$.
\end{proof}

\begin{figure}[hbt]
  \centering
  \includegraphics[width=0.6\textwidth]{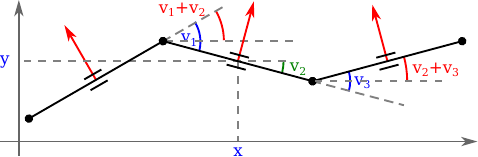}
  \caption{The non-slipping condition on wheels}
  \label{fig:snake_nonslipping}
\end{figure}

The vector field $X_1$ represents the longitudinal movement of the central segment of the snake, and $X_2$ represents the rotational movement of the central segment.

\begin{definition}
    We call a distribution $\mc{D}$ \emph{bracket-generating} if there is $r \in \N$ with the following filtration
    \begin{equation} \label{growth_vector}
    \begin{gathered}
        TM = \mc{D}^{r} \supset \dots \supset \mc{D}^{2} \supset \mc{D}^{1} \coloneqq \mc{D},\\
        \mc{D}^{i+1} \coloneqq \mc{D}^{i} + [\mc{D},\mc{D}^{i}].
    \end{gathered}
    \end{equation}

    The sequence $(\dim \mc{D}, \dim \mc{D}^2, \dots, \dim \mc{D}^r)$ is called the \emph{growth vector} of the distribution $\mc{D}$.
\end{definition}

Let us define $X_3, X_4, X_5$ by Lie brackets of $X_1, X_2$ such that
\begin{equation*}
    \begin{gathered}
        X_1,\\ X_2,
    \end{gathered}
    \quad X_3 = [X_1,X_2], \quad
    \begin{gathered}
        X_4 = [X_1, X_3],\\
        X_5 = [X_2, X_3].
    \end{gathered}
\end{equation*}
The vector fields $X_1, \dots, X_5$ are linearly independent in any point of an open and dense subset of $\mc{M}$, therefore the growth vector is $(2,3,5)$ in these points. The distribution $\mc{D}$ is bracket-generating in any point of $\mc{M}$ except the positions where an adjacent segment overlaps the central one. However, the straight positions and the positions with wheels lying on a circle have the growth vector $(2,3,4,5)$.

\begin{remark}
    Consider an open connected subset of $\mc{M}$ where $\mc{D}$ is bracket-generating. Thanks to Chow-Rashevskii theorem, we can join any two points of this subset by a path tangent to $\mc{D}$.
    
    Moreover, any non-overlapped position of the snake can be smoothly transformed to any other non-overlapped one, i.e.~ the open subset of non-overlapped positions is even path-connected, and any two such positions can be connected by flows of the vector fields $X_1, X_2$.
\end{remark}

Further, when we add the vector field
\[X_6 = [X_1,X_4],\]
we realize that $X_1, \dots, X_6$ form a Lie subalgebra of $\mfk{X}(\mc{M})$ with the multiplication table displayed in Table~\ref{multiplication_table}.

\begin{table}[htb]
\centering
 \begin{tabular}{c|cccccc}
    $[\cdot, \cdot]$ & $X_1$ & $X_2$ & $X_3$ & $X_4$ & $X_5$ & $X_6$\\ \hline
    $X_1$ & 0 & $X_3$ & $X_4$ & $X_6$ & 0 & $-X_4$ \\
    $X_2$ & $-X_3$ & 0 & $X_5$ & 0 & $4 X_3+ 4 X_6$ & 0 \\
    $X_3$ & $-X_4$ & $-X_5$ & 0 & 0 & 0 & 0 \\
    $X_4$ & $-X_6$ & 0 & 0 & 0 & 0 & 0 \\
    $X_5$ & 0 & $-4 X_3 - 4 X_6$ & 0 & 0 & 0 & 0 \\
    $X_6$ & $X_4$  & 0 & 0 & 0 & 0 & 0 
 \end{tabular}
 \caption{\rule{0pt}{15pt}
 Multiplication table of vector fields $X_1, \dots X_6$.}
 \label{multiplication_table}
\end{table}

Since the vector fields $X_1, \dots, X_5$ form a basis of $T\mc{M}$, we express $X_6$ in $X_1, \dots X_5$, obtaining%
\footnote{in \cite{md2024}, there was an error in a sign}
\begin{equation} \label{X6}
    X_6 = A \cdot X_2 - X_3 + A \cdot X_4 + B \cdot X_5.
\end{equation}
for the functions $A, B$ computed in \cite[\texttt{std\_snake.mw}]{attached_files} as
\[ A = \frac{(1 + \cos v_3) \sin v_1 - (1 + \cos v_1) \sin v_3}{\cos v_3 - \cos v_1},
\quad B = \frac{\sin v_3 \sin v_1 + \cos v_1 \cos v_3 - 1}{2\cos v_3 - 2\cos v_1}.\]

Take the coframe $\theta^1, \dots, \theta^5$ dual to the frame $X_1, \dots, X_5$, hence $\theta^i (X_j) = \delta^i_j$, where $\delta^i_j$ is the Kronecker delta. Using the formula for exterior derivatives of coframes
    \[ \di \theta^i (X_j,X_k) = -\theta^i ([X_j,X_k]), \]
the equation \eqref{X6} and Table~\ref{multiplication_table}, we find the structure equations as follows.
\begin{equation}\label{coframe_structure_standard_1}
    \begin{aligned}
    \di \theta^1 &= 0,\\
    \di \theta^2 &= -A\, \theta^1 \wedge \theta^4 - 4 A\, \theta^2 \wedge \theta^5,\\
    \di \theta^3 &= -\theta^1 \wedge \theta^2 + \theta^1 \wedge \theta^4,\\
    \di \theta^4 &= -\theta^1 \wedge \theta^3 - A\, \theta^1 \wedge \theta^4 - 4 A\, \theta^2 \wedge \theta^5,\\
    \di \theta^5 &= -\theta^2 \wedge \theta^3 - B\, \theta^1 \wedge \theta^4 - 4 B\, \theta^2 \wedge \theta^5.
    \end{aligned}\\
\end{equation}
Further, applying $[X_i,X_6]$, we readily obtain the directional derivatives $X_i(A), X_i(B)$. Since $\di A = \sum X_i(A) \theta^i$, $\di B = \sum X_i(B) \theta^i$, we receive the exterior derivatives of the functions $A,B$ that are polynomial in $A,B$,
\begin{equation} \label{coframe_structure_standard_2}
    \begin{aligned}
    \di A &= -A^2\, \theta^1 - 4 A B\, \theta^2 + 4 A^2\, \theta^5,\\
    \di B &= -A B\, \theta^1 + (1 - 4B^2)\, \theta^2 + A\, \theta^3 + 4 A B\, \theta^5.
    \end{aligned}
\end{equation}

\subsection{Symmetries} \label{section:std_snake_symmetries}
Manifolds with a $(2,3,5)$-distribution admit a description in terms of Cartan geometries, more specifically, parabolic geometries. We discuss this in detail in Section \ref{section:parabolic}. There is a unique \emph{normal regular} Cartan connection for any $(2,3,5)$-distribution \cite[pp. 431-432]{parabook}. Thus, symmetries of the Cartan connections are in one-to-one correspondence with the following concept.

\begin{definition}
    By \emph{symmetries} of the distribution $\mc{D}$, we understand the automorphisms $\phi\colon \mc{M} \to \mc{M}$ preserving the distribution $\mc{D}$, i.e.
\begin{equation*}
    \phi_* \mc{D} = \mc{D}.
\end{equation*}
For more visual interpretation, the symmetries of $\mc{D}$ are exactly automorphisms of $\mc{M}$ that map paths tangent to $\mc{D}$ (admissible robot's move) to other paths tangent to~$\mc{D}$.

\emph{Infinitesimal symmetries} are vector fields $S \in \mfk{X}(\mc{M})$ such that their flows
\begin{equation*}
    \mathrm{Fl}^S_t\colon \mc{M} \to \mc{M}
\end{equation*}
form a one-parameter family of $\mc{D}$-preserving automorphisms for $t\in (-\epsilon, \epsilon) \subset \R$. It is equivalent to the fact that their Lie brackets preserve $\mc{D}$ in the sense
\begin{equation}
    [S, X] \in \mc{D}
\end{equation}
for all $X \in \mc{D}$.
\end{definition}

Set of infinitesimal symmetries of a bracket-generating distribution $\mc{D}$
\begin{equation}
    \mc{S} = \left\{ S \in \mfk{X}(\mc{M}) \mid [S,X]\in \mc{D}\ \text{for all}\ X\in\mc{D} \right\}
\end{equation}
forms a Lie algebra and $\dim \mc{S}$ is bounded by the dimension of the structure group of the corresponding Cartan geometry \cite[Theorem 1.5.11]{parabook}, i.e.~by $\dim G_2 = 14$ in the case of $(2,3,5)$-distributions.

\begin{example}
    The obvious symmetries of the standard snake are rotations and translations of the whole robot on the plane. The corresponding infinitesimal symmetries are
    \begin{align*}
        \partial_x, \partial_y, \partial_{v_2} - y\, \partial_{x} + x\, \partial_y.
    \end{align*}
\end{example}

To find the unique Cartan connection, we start with any regular one and improve it until its curvature satisfies a certain condition. Any Cartan connection can be viewed as a matrix of one-forms built from $\theta^1, \dots, \theta^5$. Therefore, we use the equations \eqref{coframe_structure_standard_1}, \eqref{coframe_structure_standard_2}
in order to deal with the curvature. The fundamental part of the curvature of the normal regular Cartan connection can be equivalently written as a binary quartic form (\cite[p.~26]{dt2022}), which coincides with the biquartic invariant in the Cartan's paper \cite[p.~152]{ec1910}.

\begin{proposition} \label{prop:std_snake_symmetries}
    The Cartan biquartic of the standard 3-link snake is\\
{\footnotesize
    \begin{equation}
    \begin{aligned}
        \left(-\frac{128}{75}\right. &A^2 B^2 - \frac{82}{75} B^4 - \frac{16}{75} B^2- \frac{28}{75} A^4 - \frac{32}{75} A^2 + \left. \frac{8}{75} \right) x_1^4\\
        &+ 4 \cdot \left(\frac{16}{25} A^3 B+\frac{58}{25} A B^3-\frac{8}{25} A B \right) x_1^3 x_2\\
        &+ 6 \cdot \left(-\frac{566}{225} A^2 B^2 - \frac{64}{225} A^4 - \frac{104}{225} A^2 - \frac{1024}{225} B^4+ \frac{8}{225} + \frac{206}{225} B^2\right) x_1^2 x_2^2\\
        &+ 4 \cdot \left( \frac{58}{25} A^3 B + \frac{256}{25} A B^3 - \frac{38}{25} A B \right) x_1 x_2^3\\
        &+ \left( -\frac{2048}{75} A^2 B^2 - \frac{7168}{75} B^4 + \frac{2048}{75} B^2 - \frac{82}{75} A^4 - \frac{356}{75} A^2 - \frac{82}{75} \right) x_2^4,
    \end{aligned}
    \end{equation}}
\\ 
and the dimensions of the Lie algebra $\mc{S}$ of infinitesimal symmetries of $(\mc{M},\mc{D})$ is at most 5. 
\end{proposition}
\begin{proof}
    We obtain the Cartan biquartic by implementation of the procedure of Section \ref{section:parabolic} in \cite[\texttt{std\_snake.mw}]{attached_files}. The Jacobi matrix of $A,B$ has rank two, therefore the biquartic does not vanish identically, neither does its discriminant. By original Cartan's result, the non-vanishing discriminant implies that $\dim \mc{S} \leq 5$, see Theorem~\ref{prop:biquartic} below.
\end{proof}

\section{Generalized 3-link snake} \label{section:gen_snake}

Let us generalize the snake robot by considering different segment lengths $\ell_1, \ell_2, \ell_3$ and relative wheels positions $a, b, c \in [0,1]$ on segments. Without loss of generality, we can put the wheels on the adjacent segments to the end points, i.e.~$a = 0$, $c = 1$, and set the length of the central segment $\ell_2 = 1$, as discussed in \cite{tf2024}, \cite{md2024} and \cite{dt2018}.

\begin{figure}[hbt]
  \centering
  \includegraphics[width=0.6\textwidth]{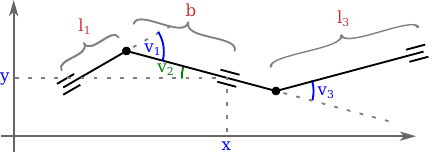}
  \caption{The 3-link snake with parameters $\ell_1,\ell_3,b$}
  \label{fig:snake_general}
\end{figure}

In the following, we will use $\ell_1, \ell_3, b$ as parameters in order to study the question posed by Paweł Nurowski in \cite{pn2014}, whether there is a choice of parameters for this robot to get a locally flat model of the correspondent parabolic geometry (i.e.~with $\mathfrak{g}_2$ as the algebra of infinitesimal symmetries) or one of the other homogeneous models which were described already in the famous Cartan's paper \cite{ec1910} and further discussed and completed in \cite{pn2005}, \cite{tw2019} and \cite{dt2022}.

The configuration space is the same as before, i.e.~$\mc{M} = S^1 \times S^1 \times S^1 \times \R^2$, and we consider the coordinates $(v_1,v_2,v_3,x,y)$ as depicted in Figure \ref{fig:snake_general}. Again, we define a distribution $\mc{D}_{\ell_1,\ell_3,b}$ from the non-slipping conditions in the wheels. For brevity, we write $\mc{D} = \mc{D}_{\ell_1,\ell_3,b}$ -- the standard $3$-link snake appears for $\ell_1 = \ell_3 = b = \frac{1}{2}$. Similarly, we choose basis vector fields for the distribution $\mc{D}$, so that they represent the longitudinal movements of the central segment and its rotational movement around the central wheel. We again denote them by $X_1,X_2$ and define as follows.

\begin{proposition}
    The distribution $\mc{D}$ admits the following vector fields $X_1,X_2$ as a basis.
    \begin{equation}\label{generators}
    \begin{gathered}
        X_1 = \left( 1 + \frac{b}{\ell_1} \cdot \cos(v_1) \right)\, \partial_{v_1} - \partial_{v_2} + \left(1 + \frac{1-b}{\ell_3} \cdot \cos(v_3)\right)\, \partial_{v_3},\\
        X_2 = -\frac{1}{\ell_1} \cdot \sin(v_1)\, \partial_{v_1} + \frac{1}{\ell_3} \cdot \sin(v_3)\, \partial_{v_3} + \cos(v_2)\, \partial_{x} + \sin(v_2)\, \partial_{y}.
    \end{gathered}
    \end{equation}

    The distribution $\mc{D}$ is bracket-generating with the growth vector $(2,3,5)$ in any generic point of $\mc{M}$ and for any choice of parameters $\ell_1,\ell_3>0, b\in [0,1]$.
\end{proposition}
\begin{proof}
    Similarly as in the proof of Proposition \ref{prop:generators_standard}, the non-slipping realizes as
    {\footnotesize
    \begin{align*}
        \tfrac{\di}{\di t} (x,y) &\perp (-\sin v_2, \cos v_2),\\
        \tfrac{\di}{\di t} \bigl( (x,y) + (1-b) (\cos v_2, \sin v_2) + \ell_3 (\cos(v_3 + v_2), \sin (v_3 + v_2)) \bigr) &\perp (-\sin (v_3+v_2), \cos (v_3 + v_2)),\\
        \tfrac{\di}{\di t} \bigl( (x,y) - b (\cos v_2, \sin v_2) - \ell_1 (\cos(v_1 + v_2), \sin (v_1 + v_2)) \bigr) &\perp (-\sin (v_1+v_2), \cos (v_1 + v_2)).
    \end{align*}}%
    These conditions correspond to $\langle Y,\varphi_1 \rangle = \langle Y,\varphi_2 \rangle = \langle Y,\varphi_3 \rangle = 0$ for any vector field $Y \in \mc{D}$ and the one-forms
    \begin{align*}
        \varphi_1 &= -\sin v_2\, \di x + \cos v_2\, \di y,\\
        \varphi_2 &= -\sin(v_3 + v_2)\, \di x + \cos(v_3 + v_2)\, \di y + \big((1-b)\cos v_3 + \ell_3 \big)\, \di v_2 + \ell_3\, \di v_3,\\
        \varphi_3 &= -\sin(v_1 + v_2)\, \di x + \cos(v_1 + v_2)\, \di y - \ell_1\, \di v_1 - (b \cos v_1 + \ell_1)\, \di v_2.
    \end{align*}

    These three one-forms are linearly independent. By direct computation, see \cite[\texttt{vectorfields.mw}]{attached_files}, we find that the vector fields $X_1,X_2$ annihilate these one-forms.

    Let us express few Lie bracket iterations of $X_1,X_2$,
    \begin{equation}\label{VF345}
        \begin{aligned}
            X_3 &= \frac{-b - \ell_1\cos v_1}{\ell_1^2}\, \partial_{v_1} + \frac{1-b + \ell_3 \cos v_3}{\ell_3^2}\, \partial_{v_3} + \sin v_2\, \partial_x - \cos v_2\, \partial_y,\\
            X_4 &= \frac{(\ell_1^2 - b^2) \sin v_1}{\ell_1^3} \partial_{v_1} + \frac{\left((1-b)^2 - \ell_3^2 \right) \sin v_3}{\ell_3^3} \partial_{v_3} - \cos v_2\, \partial_x - \sin v_2\, \partial_y,\\
            X_5 &= \frac{-\ell_1 - b\cos v_1}{\ell_1^3}\, \partial_{v_1} - \frac{\ell_3 + (1-b) \cos v_3}{\ell_3^3}\, \partial_{v_3}.
        \end{aligned}
    \end{equation}
    Substituting any generic point $u\in \mc{M}$, we find the vectors $X_1(u), \dots, X_5(u)$ linear independent, see \cite[\texttt{vectorfields.mw}]{attached_files}, hence the distribution $\mc{D}$ is bracket-generating with the growth vector $(2,3,5)$ for any $\ell_1,\ell_3>0$, $b \in [0,1]$ on a Zariski open subset of $\mc{M}$.
\end{proof}

Therefore, the distribution $\mc{D}$ has the growth vector $(2,3,5)$ in any generic point of $\mc{M}$ and
\begin{equation}\label{frame}
    X_1,\quad X_2,\quad X_3 = [X_1, X_2], \quad X_4 = [X_1,X_3], \quad X_5 = [X_2,X_3]
\end{equation}
form a frame of $T\mc{M}$.
\begin{lemma} \label{lemma:symmetry}
    The distributions $\mc{D}_{\ell_1,\ell_3,b}$ are equivalent to $\mc{D}_{\ell_3,\ell_1,1-b}$, i.e.~there is an automorphism $\phi$ on $\mc{M}$ such that $\phi_* (\mc{D}_{\ell_1,\ell_3,b}) = \mc{D}_{\ell_3,\ell_1,1-b}$.
\end{lemma}
\begin{proof}
    Intuitively, one can see it directly from Figure \ref{fig:snake_general} and easily verify using the expressions \eqref{generators} that the push-forward of
    \[ \phi\big((v_1,v_2,v_3,x,y)\big) = (v_3,v_2,v_1,-x,-y) \]
    maps the basis $X_1, X_2 \in \mc{D}_{\ell_1, \ell_3, b}$ to the basis $X_1, -X_2 \in \mc{D}_{\ell_3, \ell_1, 1-b}$.
\end{proof}
\begin{observation}
Both vector fields in \eqref{generators} are of the form
\begin{equation}\label{VF_closed_form}
\begin{aligned}
    (a_1 \cdot \cos v_1 &+ a_2 \cdot \sin v_1 + a_3) \,\partial_{v_1} + a_4 \,\partial_{v_2} + (a_5 \cdot \cos v_3 + a_6 \cdot \sin v_3 + a_7) \,\partial_{v_3}\\
    &+ (a_8 \cdot \cos v_2 + a_9 \sin v_2) \,\partial_x + (a_{10} \cdot \cos v_2 + a_{11} \sin v_2) \,\partial_y,
\end{aligned}
\end{equation}
where $a_i$ are real parameters.
\end{observation}

\begin{lemma} \label{lemma:Lie_bracket_closed}
    The set of vector fields of the form \eqref{VF_closed_form} is closed under the Lie bracket of vector fields, hence it is an $11$-dimensional Lie subalgebra $\mfk{u}$ of $\mfk{X}(\mc{M})$.
\end{lemma}
\begin{proof}
    Let us proof the claim for the $\partial_{v_1}$-term, the proof for $\partial_{v_3}$-term is the same and for the $\partial_{v_2}, \partial_{x}, \partial_{y}$-terms is obvious.

    Consider two vector fields $Y_1$ and $Y_2$ of the form \eqref{VF_closed_form} with the  coefficients $a_i$ and $b_i$.
    \begin{align*}
        \di v_1 ([Y_1,Y_2])
        &= (a_1 \cos v_1 + a_2 \sin v_1 + a_3) \cdot \partial_{v_1}(b_1 \cos v_1 + b_2 \sin v_1 + b_3)\\
        &\qquad - (b_1 \cos v_1 + b_2 \sin v_1 + b_3) \cdot \partial_{v_1}(a_1 \cos v_1 + a_2 \sin v_1 + a_3)\\
        &= (a_1 \cos v_1 + a_2 \sin v_1 + a_3) \cdot (-b_1 \sin v_1 + b_2 \cos v_1)\\
        &\qquad - (b_1 \cos v_1 + b_2 \sin v_1 + b_3) \cdot (-a_1 \sin v_1 + a_2 \cos v_1)\\
        &= \bigl( (-a_1 b_1 + a_2 b_2) \sin v_1 \cos v_1 - a_2 b_1 \sin^2 v_1 + a_1 b_2 \cos^2 v_1\\
        &\qquad\qquad - a_3 b_1 \sin v_1 + a_3 b_2 \cos v_1) \bigr) \\
        &\qquad - \bigl( (-b_1 a_1 + b_2 a_2) \sin v_1 \cos v_1 - b_2 a_1 \sin^2 v_1 + b_1 a_2 \cos^2 v_1\\
        &\qquad\qquad - b_3 a_1 \sin v_1 + b_3 a_2 \cos v_1) \bigr) \\
        &= 0 \sin v_1 \cos v_1 - a_2 b_1 (\sin^2 v_1 + \cos^2 v_1) + a_1 b_2 (\cos^2 v_1 + \sin^2 v_1)\\
        &\qquad\qquad + (-a_3 b_1 + b_3 a_1) \sin v_1 + (a_3 b_2 - b_3 a_2) \cos v_1\\
        &= (-a_3 b_1 + b_3 a_1) \sin v_1 + (a_3 b_2 - b_3 a_2) \cos v_1 + (-a_2 b_1 + a_1 b_2).
        \qedhere
    \end{align*}
\end{proof}

In a similar manner to the definition of growth vectors for distributions \eqref{growth_vector}, we define growth vectors for Lie algebras.
\begin{definition}
    Consider a Lie algebra $\mfk{g}$ and a set of its elements $\mfk{d}$. Denote
    \begin{equation}
        \mfk{d}^1 = \mfk{d},
        \quad \mfk{d}^{i+1} = \mfk{d}^i \cup \{ [X,Y] \mid X\in \mfk{d}, Y \in \mfk{d}^i \},
        \quad \mfk{D}^{i} = \on{span}_\R \mfk{d}^i \subset \mfk{X}(\mc{M}),
    \end{equation}
    and $r \in \N$ the least number such that $\mfk{D}^r = \mfk{D}^{r+1}$. For short, we will denote vector fields 
    obtained by iterated Lie brackets by
    \begin{equation} \label{LB_shortcut}
        X_{i_1 i_2\dots i_{k-1} i_k} \coloneqq [X_{i_1}, [X_{i_2}, [\cdots [X_{i_{k-1}}, X_{i_k}]\cdots]]].
    \end{equation}
    
    We call the sequence
    \begin{equation} \label{growth_vector_by_VF}
        (\dim \mfk{D}^1, \dots, \dim \mfk{D}^r)
    \end{equation}
    \emph{the growth vector of} $\mfk{g}$ \emph{with respect to} $\mfk{d}$.
\end{definition}

\subsection{Branching}

For $(\ell_1,\ell_3,b) \in \R^+ \times \R^+ \times [0,1]$, any Lie algebra $\mfk{s}_{\ell_1,\ell_3,b}$ generated by $X_1,X_2 \in \mc{D}_{\ell_1,\ell_3,b}$ is the subalgebra of the Lie algebra
\[ \mfk{u} = \{ X \in \mfk{X} (\mc{M}) \mid X\ \text{is of the form \eqref{VF_closed_form}}\}, \]
therefore $\dim \mfk{s}_{\ell_1,\ell_3,b} \leq \dim \mfk{u} = 11$.

We are going to characterize all the Lie subalgebras $\mfk{s}_{\ell_1,\ell_3,b} \subset \mfk{u}$ in terms of our choice of basis vector fields. We are not interested whether some of the Lie algebras $\mfk{s}_{\ell_1,\ell_3,b}$ are isomorphic, our intention is to find a multiplication table for families of $\mfk{s}_{\ell_1,\ell_3,b}$ that would depend on parameters $\ell_1,\ell_3,b$ in order to reproduce the procedure of Section~\ref{section:std_snake}.

Let $\mfk{d} = \{X_1,X_2\}$ and $\mfk{d}^i, \mfk{D}^i$, $r$ be defined as in the Definition above. Using the notation \eqref{LB_shortcut}, we fix the vector fields from $\mfk{d}^i$ in the following order
\begin{equation} \label{VF_order}
    \underbrace{X_1,\dots,X_5}_{\mfk{d}^3}, \underbrace{X_{14} ,X_{25}}_{\mfk{d}^4\setminus \mfk{d}^3}, \underbrace{X_{114}, X_{214}, X_{125}, X_{225}}_{\mfk{d}^5 \setminus \mfk{d}^4}, \underbrace{X_{1114}, X_{2214}, X_{1125}, X_{2225}}_{\mfk{d}^6 \setminus \mfk{d}^5}, \dots.
\end{equation}
In \eqref{VF_order}, we omit the following vector fields, since we directly compute that they vanish, see \cite[\texttt{vectorfields.mw}]{attached_files}.
\[ 0 = X_{15} = X_{24} = X_{2114} = X_{1214} = X_{1225} = X_{2125}. \]

The vector fields \eqref{VF_order} generate the Lie algebra $\mfk{s}_{\ell_1, \ell_3, b}$ and we want to choose a basis from them. We proceed from left to right in that order and we determine for which parameters the current vector field is linearly independent on the previously selected independent ones. In each step, it splits the set of parameters $\ell_1, \ell_3, b$ by the assumed dependence or independence into branches, and we gradually find bases of vector spaces
\[\on{span}_\R \mfk{d} = \mfk{D}^1 \subset \dots \subset \mfk{D}^r = \mfk{D}^{r+1} = \mfk{s}_{\ell_1,\ell_3,b} \]
for each branch. In other words, we form a matrix with $a_i$ coordinates \eqref{VF_closed_form} of $X_1,\dots, X_5, X_{14}, \dots$ in columns and we determine its rank with respect to parameters $\ell_1, \ell_3, b$.

We start with parameters $(\ell_1,\ell_3,b) \in \R^+ \times \R^+ \times [0,1]$. We already know that
\[ \mfk{D}^3 = \on{span}_\R \mfk{d}^3 = \on{span}_\R \{X_1, X_2, \dots, X_5\},\qquad \dim \mfk{D}^3 = 5,\]
and we continue with vector fields $X_{14}, X_{25} \in \mfk{d}^4 \setminus \mfk{d}^3$. The vector field $X_{14}$ is always linearly independent on $\mfk{d}^3$. Considering $X_{25}$, the set of parameters $\R^+ \times \R^+ \times [0,1]$ splits into branches with basis of $\mfk{D}^4$ formed by
\begin{enumerate}[label=\arabic*)]
    \item $X_1, \dots, X_5, X_{14}$ for $b=\frac{1}{2}$, $\ell_1$ and $\ell_3$ arbitrary \label{branch1}
    \item $X_1, \dots, X_5, X_{14}$ for $\ell_1 = b = 1 - \ell_3$, $b\ne0,1$ arbitrary \label{branch2}
    \item $X_1, \dots, X_5, X_{14}, X_{25}$ otherwise. \label{branch3}
\end{enumerate}

Further, consider
\[ \mfk{d}^5 \setminus \mfk{d}^4 = \{X_{114}, X_{214}, X_{125}, X_{225}\}. \]
For branches \ref{branch1} and \ref{branch2}, we only need to distinguish linear dependency of $X_{114}, X_{214}$, the vector fields $X_{125}, X_{225}$ will not appear since $X_{25}$ is already linearly dependent on the chosen basis of $\mfk{D}^4$. Regarding branch \ref{branch3}, we need to check all vector fields from $\mfk{d}^5 \setminus \mfk{d}^4$. We find that the basis of branch \ref{branch2} is closed with respect to Lie bracket, $\mfk{D}^4 = \mfk{D}^5$, hence it is a basis of $\mfk{s}_{\ell_1, \ell_3, b}$. Continuing, we obtain several new branches, all but one are closed for $\mfk{D}^5$, the last one is closed for $\mfk{D}^6$. The final bases of $\mfk{s}_{\ell_1, \ell_3, b}$ for all branches are shown schematically in Figure~\ref{fig:branches}.

\begin{figure}[hbt]
    \centering
    \begin{equation*}
        X_1, \dots, X_{14}:
        \begin{cases}
            {b = \frac{1}{2}}& :
            \begin{cases}
                {\ell_1 = \ell_3}& :\emptyset\\
                {\ell_1 \neq \ell_3}& :X_{114}, X_{214} :
                \begin{cases}
                    {\ell_1 = \frac{1}{2}}& :\emptyset\\
                    {\ell_3 = \frac{1}{2}}& :\emptyset\\
                    {\text{otherwise}}& :X_{1114}
                \end{cases}
            \end{cases}\\
            {\ell_1 = b = 1 - \ell_3}& :\emptyset\\
            {\text{otherwise}}& :X_{25} :
            \begin{cases}
                {b=1}& :
                \begin{cases}
                    {\ell_1 = 1}& :X_{125}\\
                    {\ell_1 \neq 1}& :X_{214}, X_{125}
                \end{cases}\\
                {b = 0}& :
                \begin{cases}
                    {\ell_3 = 1}& :X_{125}\\
                    {\ell_3 \neq 1}& :X_{214}, X_{125}
                \end{cases}\\
                {b = \frac{\ell_1}{\ell_1 + \ell_3}}& :X_{125}, X_{225}\\
                
                {\text{otherwise}}& :X_{114} :
                \begin{cases}
                    {b = \ell_1}& :\emptyset\\
                    {b = 1 - \ell_3}& :\emptyset\\
                    {\text{otherwise}}& :X_{214}
                \end{cases}
            \end{cases}
        \end{cases}
    \end{equation*}%
    \caption{\centering Basis of $\mfk{s}_{\ell_1, \ell_3, b}$ by branches of parameters.}
    \label{fig:branches}
\end{figure}

Below, we list the bases of $\mfk{s}_{\ell_1,\ell_3,b}$ for each branch together with their growth vector with respect to $\{X_1,X_2\}$, defined by\eqref{growth_vector_by_VF}. Recall the considered range of parameters
\[P = \R^+ \times \R^+ \times [0,1].\]
The following branches are ordered such that a branch with $X_\alpha$ linearly dependent comes before branches with $X_\alpha$ contained in their bases.
We reduce the number of branches pictured in Figure~\ref{fig:branches} by the equivalence $(\ell_1, \ell_3,b) \sim (\ell_3, \ell_1, 1-b)$ from Lemma \ref{lemma:symmetry}.
\begin{enumerate}[label=\arabic*)]
    \item $\mfk{s}_{\ell_1,\ell_3,b} = \langle X_1, \dots, X_{14} \rangle$ with the growth vector $(2,3,5,6)$ for
        \begin{equation*}
        \begin{aligned}
            a)&\quad \left\{ (\ell_1,\ell_3,b) \in P \mid b = \tfrac{1}{2},\ \ell_1 = \ell_3 \right\}\\
            b)&\quad \left\{ (\ell_1,\ell_3,b) \in P \mid \ell_1 = b = 1 - \ell_3 \right\}.
        \end{aligned}
        \end{equation*}\label{branch_list_1}
    \item $\mfk{s}_{\ell_1,\ell_3,b} = \langle X_1, \dots, X_{14}, X_{114}, X_{214} \rangle$ with the growth vector $(2,3,5,6,8)$ for
        \begin{equation*}
            \left\{ (\ell_1,\ell_3,b) \in P \mid b = \ell_1 = \tfrac 1 2 \neq \ell_3 \right\}.
        \end{equation*}
        The equivalent branch is
        \[ \left\{ (\ell_1,\ell_3,b) \in P \mid b = \ell_3 = \tfrac 1 2 \neq \ell_1 \right\}. \]
    \item $\mfk{s}_{\ell_1,\ell_3,b} = \langle X_1, \dots, X_{14}, X_{114}, X_{214}, X_{1114} \rangle$ with the growth vector $(2,3,5,6,8,9)$ for
        \begin{equation*}
            \left\{ (\ell_1,\ell_3,b) \in P \mid b = \tfrac{1}{2},\ \ell_1 \neq \ell_3,\ \ell_1\ne b,\ \ell_3\ne b \right\}.
        \end{equation*}
    \item $\mfk{s}_{\ell_1,\ell_3,b} = \langle X_1, \dots, X_{14}, X_{25}, X_{125} \rangle$ with the growth vector $(2,3,5,7,8)$ for
        \begin{equation*}
            \left\{ (\ell_1,\ell_3,b) \in P \mid b = \ell_1 = 1 \right\}.
        \end{equation*}
        The equivalent branch is
        \[ \left\{ (\ell_1,\ell_3,b) \in P \mid b = 0,\ \ell_3 = 1 \right\}. \]
    \item $\mfk{s}_{\ell_1,\ell_3,b} = \langle X_1, \dots, X_{14}, X_{25}, X_{125}, X_{225} \rangle$ with the growth vector $(2,3,5,7,9)$ for
        \begin{equation*}
            \left\{ (\ell_1,\ell_3,b) \in P \mid b = \tfrac{\ell_1}{\ell_1+\ell_3},\ \ell_1 \neq \ell_3,\ \ell_1 + \ell_3 \neq 1 \right\}.
        \end{equation*}
    \item $\mfk{s}_{\ell_1,\ell_3,b} = \langle X_1, \dots, X_{14}, X_{25}, X_{214}, X_{125} \rangle$ with the growth vector $(2,3,5,7,9)$ for
        \begin{equation*}
            \left\{ (\ell_1,\ell_3,b) \in P \mid b = 1, \ell_1 \neq 1 \right\}.
        \end{equation*}
        The equivalent branch is
        \[ \left\{ (\ell_1,\ell_3,b) \in P \mid b = 0,\ \ell_3 \neq 1 \right\}. \]
    \item $\mfk{s}_{\ell_1,\ell_3,b} = \langle X_1, \dots, X_{14}, X_{25}, X_{114} \rangle$ with the growth vector $(2,3,5,7,8)$ for
        \begin{equation*}
            \left\{ (\ell_1,\ell_3,b) \in P \mid b = \ell_1,\ \ell_1 + \ell_3 \neq 1,\ b \neq \tfrac{1}{2}, 1 \right\}.
        \end{equation*}
        The equivalent branch is
        \[ \left\{ (\ell_1,\ell_3,b) \in P \mid b = 1 - \ell_3,\ \ell_1 + \ell_3 \neq 1,\ b \neq 0, \tfrac 1 2 \right\}. \]
    \item $\mfk{s}_{\ell_1,\ell_3,b} = \langle X_1, \dots, X_{14}, X_{25}, X_{114}, X_{214} \rangle$ with the growth vector $(2,3,5,7,9)$ for
        \begin{equation*}\label{item:branch_generic}
            \quad \left\{ (\ell_1,\ell_3,b) \in P \mid b \neq 0,\tfrac{1}{2},1,\ell_1,1 - \ell_3,\tfrac{\ell_1}{\ell_1+\ell_3} \right\}.
        \end{equation*} 
        This branch is the generic case.
\end{enumerate}

\begin{proposition}\label{prop:branches}
    The system of vector fields \eqref{generators} generates a finite dimensional Lie subalgebra $\mfk{s}_{\ell_1,\ell_3,b}$ of $\mathfrak{X}(\mc{M})$ over the field $\R$. Depending on the parameters $\ell_1,\ell_3,b$, the Lie algebra $\mfk{s}_{\ell_1,\ell_3,b}$ has dimension $6$, $8$ or $9$. Basis of $\mfk{s}_{\ell_1,\ell_3,b}$ for any parameters $(\ell_1, \ell_3, b)$ belong to one branch listed above or to the equivalent one.
\end{proposition}
\begin{proof}
    Iteratively written Lie brackets of $X_1,X_2$ in $\mfk{d}, \mfk{d}^2, \mfk{d}^3\dots$ span the Lie algebra $\mfk{s}_{\ell_1, \ell_3, b}$ and we want to choose a basis from them. Consistently with the procedure described above, we rewrite the vector fields from $\mfk{d}, \mfk{d}^2, \dots, \mfk{d}^7$ in the order \eqref{VF_order} into the form \eqref{VF_closed_form} and put their $11$ coefficients $a_i$ into columns of a matrix. Using Gaussian elimination on the matrix, we gradually choose linearly independent column vectors from left to right in the order \eqref{VF_order} with respect to the parameters $\ell_1, \ell_3, b$. As soon as all columns corresponding to $\mfk{d}^{k+1} \setminus \mfk{d}^k$ are linearly dependent on the previous columns for some index $k$, we obtain a basis of $\mfk{s}_{\ell_1, \ell_3, b}$ for the branch of parameters specified during the elimination.
    
    The matrix is explicitly computed in \cite[\texttt{vectorfields.mw}]{attached_files} and copied to \cite[\texttt{snake\_coefficient\_matrix.py}]{attached_files} in a suitable form for \texttt{SymPy}. The Gaussian elimination is performed in \texttt{Jupyter} nootebook \cite[\texttt{parameter\_depended\_basis.ipynb}]{attached_files}.
\end{proof}

\begin{remark}
    We emphasize that the preceding proposition is tied to our fixed choice of the vector fields $X_1,X_2$ as in \eqref{generators}. We don't claim that those families of Lie algebras are unique. We don't know, whether there are different vector fields spanning $\mc{D}$ in one branch that generate a Lie algebra appearing for another branch.
\end{remark}

%
Although the expression of $X_1, X_2$ in coordinates $(v_1, v_2, v_3, x, y)$ was handy for finding the finite dimensional Lie algebras, the following polynomial form will be better for calculations.
\begin{lemma}
    There are coordinates $(p_1, p_2, p_3, p_4, p_5)$, in which the vector fields $X_1, X_2$ are of the form
    \begin{equation}\label{polynomial_generators}
    \begin{aligned}
        X_1 &= \frac{1}{2 \ell_1} \left( ( \ell_1 - b)\, p_1^2 + ( \ell_1 + b) \right)\, \partial_{p_1} - \partial_{p_2}\\
         &\qquad {}+ \frac{1}{2 \ell_3} \left( (\ell_3+b-1)\, p_3^2 + (\ell_3-b+1) \right)\, \partial_{p_3} - p_5\, \partial_{p_4} + p_4\, \partial_{p_5},\\
        X_2 &= -\frac{p_1}{\ell_1}\, \partial_{p_1} + \frac{p_3}{\ell_3}\, \partial_{p_3} + \partial_{p_4}.
    \end{aligned}
    \end{equation}
\end{lemma}
\begin{proof}
    We use the transformation
    \begin{equation}\label{polynomial_coordinates}
    \begin{gathered}
        p_1 = \tan \left(\tfrac{v_1}{2}\right),\quad p_2 = v_2, \quad p_3 = \tan \left(\tfrac{v_3}{2}\right)\\
        p_4 = x \cos(v_2) + y \sin(v_2),\quad p_5 = y \cos(v_2) - x \sin(v_2),
    \end{gathered}
    \end{equation}
    and push-forward the vector fields \eqref{generators} into the form \eqref{polynomial_generators} by direct computation provided in \cite[\texttt{vectorfields.mw}]{attached_files}.
\end{proof}

For instance, dealing with the standard snake $b=\ell_1=\ell_3=\tfrac{1}{2}$, the vector fields read as
    \begin{gather*}
        X_1 = \partial_{p_1} - \partial_{p_2} + \partial_{p_3} - p_5\, \partial_{p_4} + p_4\, \partial_{p_5},\\
        X_2 = -2p_1\, \partial_{p_1} + 2p_3\, \partial_{p_3} + \partial_{p_4}.
    \end{gather*}

For the same purpose as we expressed $X_6$ for the standard snake \eqref{X6}, we express any vector field $X_a$ complementary to the frame $X_1, \dots X_5$ in a basis of $\mfk{s}_{\ell_1, \ell_3,b}$, i.e.
\begin{equation}\label{coefficient_functions}
    X_a = \sum_{i=1}^5 A_a^i\, X_i,\quad A_a^i \in C^\infty(\mc{M}).
\end{equation}
Similarly to \eqref{X6}, there are relations between $A_a^1, \dots, A_a^4$, but no simple relation for $A_a^5$.
\begin{lemma} \label{lemma:Aai_relations}
    Let $X_a$ be any complementary vector field contained in a basis of $\mfk{s}_{\ell_1, \ell_3,b}$ listed above Proposition \ref{prop:branches}. Its coefficient functions $A_a^i$ from \eqref{coefficient_functions} satisfy
    \begin{equation} \label{Aai_relations}
        A_a^1 = 0,\quad
        A_a^2 = 
        \begin{cases}
            A_a^4 + 1, & a = 114 \\
            A_a^4, & \text{otherwise}
        \end{cases},\quad
        A_a^3 =
        \begin{cases}
            -1, & a = 14\\
            1, & a=1114\\
            0, & a=\text{otherwise}
        \end{cases}.
    \end{equation}
\end{lemma}
\begin{proof}
    Because the $\partial_{v_2}$ term has a constant coefficient in the expressions \eqref{generators} of $X_1$ and cannot appear in the vector field $X_a$, the equality $A_a^1 = 0$ is obvious. Looking to vector fields $X_1, \dots, X_5$ in \eqref{generators}, \eqref{VF345}, the terms with $\partial_x$ and $\partial_y$ appear in $X_2,X_3,X_4$. They contain $\pm\cos v_2$, $\pm\sin v_2$ and similar terms with coefficients $\pm 1$ or $0$ appear for $X_a$. There are three possible $\partial_x$ and $\partial_y$-terms appearing for $X_a$:
    \begin{enumerate}[label=\arabic*)]
        \item $\pm \cos v_2\,\partial_x \pm \sin v_2\,\partial_y$, then $A^2_a = A^4_a \pm 1$ and $A^3_a = 0$
        \item $0\,\partial_x + 0\,\partial_y$, then $A^2_a = A^4_a$ and $A^3_a = 0$
        \item $\pm \sin v_2\,\partial_x \mp \cos v_2\,\partial_y$, then $A^2_a = A^4_a and A^3_a = \pm 1$.
    \end{enumerate}
    The concrete values are computed in \cite[\texttt{vectorfields.mw}]{attached_files}.
\end{proof}

\section{Coframe structure equations} \label{section:coframe_structure}
In this section, we derive equations analogous to \eqref{coframe_structure_standard_1} and \eqref{coframe_structure_standard_2} that will be used later to compute the Cartan biquartics for the branches listed in Proposition \ref{prop:branches}. We will proceed in a general setting with the intended application to those branches on mind.

Let us consider vector fields $X_1, \dots, X_{n+m}$ on an $n$-dimensional manifold $\mc{M}^n$, such that $X_1,\dots, X_n$ form a frame of $\mc{M}^n$ and they generate a finite dimensional Lie algebra $\mfk{s} \subset \mfk{X}(\mc{M}^n)$ with the basis $X_1, \dots, X_n,X_{n+1}, \dots, X_{n+m}$. Denote $\theta^1,\dots,\theta^n$ the one-forms that are dual to the frame $X_1,\dots,X_n$.

We will use indices $i,j,k$ for $1, \dots, n$ and $a,b,c$ for $n+1,\dots,n+m$. The indices $\alpha,\beta,\gamma$ will be used to represent any of the listed indices. Because $X_1, \dots, X_n$ generates the $T\mc{M}^n$ in each point, we can use them to express $X_{n+1}, \dots, X_{n+m}$ in functions $A_a^k \in C^\infty (\mc{M}^n)$,
\begin{equation} \label{Aak_functions}
    X_a = \sum A_a^k \,X_k .
\end{equation}
Next, let us denote by $C_{\alpha \beta}^\gamma$ 
the structure coefficients of the considered Lie algebra.
\begin{equation} \label{structure_coefficients}
    [X_i,X_j] = \sum C_{ij}^k X_k + \sum C_{ij}^b X_b, \quad [X_i,X_a] = \sum C_{ia}^k X_k + \sum C_{ia}^b X_b.
\end{equation}

Now, we will use the fact that the Lie algebra is finitely generated and describe the differential of function coefficients $A_a^k$.
\begin{proposition}
    The exterior derivative of any $A_a^k$ is polynomial in $A_b^j$ and it holds
    \begin{equation}\label{coefficient_function_differential}
        \di A_a^k = \sum_i \left(
        C_{ia}^k + \sum_b C_{ia}^b A_b^k - \sum_j C_{ij}^k A_a^j - \sum_j \sum_b C_{ij}^b A_a^j A_b^k
        \right) \theta^i.
    \end{equation} 
\end{proposition}
\begin{proof}
    \begin{align*}
        [X_i,X_a] &= \sum C_{ia}^k \,X_k + \sum C_{ia}^b \,X_b = \sum C_{ia}^k \,X_k + \sum \sum C_{ia}^b A_b^j \,X_j\\
        &= \sum_k \left( C_{ia}^k + \sum C_{ia}^b A_b^k \right) X_k
    \end{align*}
    \begin{align*}
        [X_i, X_a] &= [X_i, \sum A_a^j \,X_j] = \sum_j \left( X_i A_a^j \right) X_j + \sum_j A_a^j \,[X_i,X_j]\\
        &= \sum_j \left( X_i A_a^j \right) X_j + \sum_j A_a^j \left( \sum_k C_{ij}^k X_k + \sum_b C_{ij}^b X_b \right)\\
        &= \sum_j \left( X_i A_a^j \right) X_j + \sum_j A_a^j \left( \sum_k C_{ij}^k X_k + \sum_b C_{ij}^b \sum_k A_b^k X_k \right)\\
        &= \sum_k \left( (X_i A_a^k) + \sum_j A_a^j C_{ij}^k + \sum_j \sum_b A_a^j C_{ij}^b A_b^k \right) X_k
    \end{align*}
    where $(X_i A_a^k)$ denote the differentiation by $X_i$. Because the two equations are equal, we can express this term by comparison of the coefficients
    \begin{equation} \label{XiAak}
        (X_i A_a^k) = \di A_a^k (X_i) = C_{ia}^k + \sum_b C_{ia}^b A_b^k - \sum_j C_{ij}^k A_a^j - \sum_j \sum_b C_{ij}^b A_a^j A_b^k. \qedhere
    \end{equation}
\end{proof}
Now, we can write the structure equations analogously to \eqref{coframe_structure_standard_1}, \eqref{coframe_structure_standard_2}.
\begin{proposition}
    The structure equations of the coframe $\theta^1, \dots, \theta^n$ are
    \begin{equation}\label{structure_equation}
        \di \theta^i = -\sum_{j<k} \left( C_{jk}^i + \sum_a C_{jk}^a A_a^i \right)\, \theta^j \wedge \theta^k.
    \end{equation}
\end{proposition}
\begin{proof}
    The equation \eqref{structure_equation} comes from the fact
    \[ \di \theta^i (X_j,X_k) = X_j \cdot \theta^i(X_k) - X_k \cdot \theta^i(X_j) - \theta^i([X_j,X_k]) = - \theta^i([X_j,X_k]), \]
    where $X_j, X_k$ are the vector fields from the frame dual to the coframe $\theta^i$. Because
    \[ [X_j,X_k] = \sum_{\alpha} C_{jk}^\alpha X_\alpha = \sum_{l} C_{jk}^l X_l + \sum_{a} C_{jk}^a X_a = \sum_{l} C_{jk}^l X_l + \sum_{a} C_{jk}^a \sum_l A_a^l X_l, \]
    the coefficient in $X_i$-term is precisely $C_{jk}^i + \sum_a C_{jk}^a A_a^i$.
\end{proof}

\section{Description via Parabolic Geometry} \label{section:parabolic}
First, we will recall key ideas about parabolic geometries from \cite[pp. 234-271]{parabook} without proofs. Using \cite[p.~431]{parabook} and \cite{dt2022}, we will specify these ideas to distributions with the growth vector $(2,3,5)$, and we will follow the comments in \cite[p.~77]{parabook} to treat them from the local point of view. We will present Algorithm~\eqref{normalization_algorithm} finding
the unique Cartan connection and analyse the harmonic curvature in Section~\ref{section:harmonic_curvature}. The algorithm is then realized for each considered distribution $\mc{D}_{\ell_1,\ell_3,b}$ in \cite[\texttt{normalization.mw}]{attached_files} and the results are presented in Section~\ref{section:results}.

It is well-known that geometries given by $(2,3,5)$-distributions admit regular, normal Cartan connections of parabolic geometries of type $(G_2,P)$, as one can find in \cite[pp.\! 271, 431]{parabook}, where $G_2$ is the $14$-dim exceptional Lie group and $P$ its parabolic Lie subgroup. The Lie algebras $\mfk{g} \coloneqq \on{Lie}(G_2)$ and $\mfk{p} \coloneqq \on{Lie}(P)$ correspond to the Dynkin diagram
\begin{equation}\label{dynkin_g2}
    \dynkin[reverse arrows=true, root radius=0.1cm, edge length=1cm]{G}{xo},
\end{equation}
where the cross $\times$ on the left simple root $\alpha$ indicate that the root space $-\alpha$ does not belong to $\mfk{p}$.

A parabolic geometry of type $(G_2,P)$ on $\mc{M}$ is a principal bundle $p \colon \mc{G} \to \mc{M}$ with the structure group $P$ together with a \emph{Cartan connection} $\omega$, i.e.~a $\mfk{g}$-valued one form $\omega \in \Omega^1(\mc{G},\mfk{g})$ such that
\begin{enumerate}[label=\arabic*)]
    \item $\omega$ is $P$-equivariant, i.e.~$R_g^* \omega = {\on{Ad}_{g^{-1}}} \circ \omega$ for all $g \in P$,
    \item it defines absolute parallelism, $\omega_u \colon T_u\mc{G} \to \mfk{g}$ is an isomorphism for all $u\in \mc{G}$,
    \item it reproduces fundamental vector fields, $\omega\left( \tfrac{d}{dt}\big|_0 R_{\on{exp} tA}(u)\right) = A$ for all $A \in \mfk{g}, u \in \mc{G}$.
\end{enumerate}

The \emph{curvature} $\Omega \in \Omega^2(\mc{G},\mfk{g})$ of the Cartan connection $\omega$ is defined by
\begin{equation} \label{curvature}
    \Omega(X,Y) \coloneqq \di\omega(X,Y) + [\omega(X),\omega(Y)]
\end{equation}
for any $X,Y \in T\mc{G}$, and is horizontal, i.e.~$\Omega(X,\cdot) = 0$ whenever $X \in V\mc{G}$. Therefore, we can define the \emph{curvature function} $\kappa\colon \mc{G} \to \bigwedge^2 (\mfk{g}/\mfk{p})^* \otimes \mfk{g}$ of $\omega$ by
\begin{equation} \label{curvature_function}
    \kappa (u) (A,B) = \Omega\left( \omega_u^{-1}(A),\omega_u^{-1}(B) \right)
\end{equation}
using any representative of $A,B \in \mfk{g}/\mfk{p}$.

\subsection{Filtration and grading}
\begin{definition}
    Let $V$ and $W$ be filtered vector spaces $V^i \supseteq V^{i+1}$, $W^j \supseteq W^{j+1}$. We say that a linear map
    \[ \textstyle \alpha\colon \bigwedge^k V \to W \]
    is of \emph{homogeneity} $\geq \ell$, if
    \begin{equation}\label{def_homogeneity_geq}
        \alpha (V^{i_1} \wedge \dots \wedge V^{i_k}) \subseteq W^{\ell + (i_1+\dots+i_k)}
    \end{equation}
    for all $i_1, \dots, i_k$. Further, fix a grading on $V$ and $W$
    \begin{equation*}
        \textstyle V = \bigoplus_i V_i,\ W = \bigoplus_i W_i,\ \text{such that}\ V^j = \bigoplus_{i \geq j} V_i,\ W^j = \bigoplus_{i \geq j} W_i.
    \end{equation*}
    We say that the linear map $\alpha$ is of \emph{homogeneity} $\ell$, if
    \begin{equation*}
        \alpha (V_{i_1} \wedge \dots \wedge V_{i_k}) \subseteq W_{\ell + (i_1+\dots+i_k)}
    \end{equation*}
    for all $i_1, \dots, i_k$. We can decompose any linear map $\beta$ between two graded spaces by its homogeneities, $\beta = \sum_\ell \beta_\ell$. Thanks to this decomposition, we get the filtration and grading on the space of linear maps $L(\bigwedge^k V,W)$ by homogeneities
    \begin{equation} \label{linear_map_homogeneities}
        \textstyle L(\bigwedge^k V,W) = \bigoplus_i L(\bigwedge^k V,W)_i, \quad
    L(\bigwedge^k V,W)^j = \bigoplus_{i \geq j} L(\bigwedge^k V,W)_i.
    \end{equation}
    Similarly, we define \emph{homogeneity} on vector bundles. 
    
    We call the vector space
    \[ \on{gr} V = \bigoplus_i (\on{gr} V)_i \coloneqq \bigoplus_i V^i/V^{i+1} \]
    the \emph{associated graded} space to $V$. When we choose representatives of a $\on{gr} V$ basis in $V$, they generate corresponding subspaces $V_i$ that form a grading. This choice is equivalent to the choice of an isomorphism of vector spaces $\on{gr} V \cong V$. Clearly, a linear map $\alpha$ of homogeneity $\geq \ell$, as in \eqref{def_homogeneity_geq}, defines an induced linear map
    \[ \widetilde{\alpha}\colon {\textstyle \bigwedge^k} \on{gr} V \to \on{gr} W  \]
    of homogeneity $\ell$.
\end{definition}

The Lie algebra $\mfk{g} = \on{Lie}(G_2)$ admits a filtration
\begin{equation}\label{filtration_on_g}
    \mfk{g} = \mfk{g}^{-3} \supset \mfk{g}^{-2} \supset \mfk{g}^{-1} \supset \mfk{g}^{0} \supset \mfk{g}^{1} \supset \mfk{g}^{2} \supset \mfk{g}^{3} \supset \{0\}
\end{equation}
where $[\mfk{g}^i,\mfk{g}^j] \subseteq \mfk{g}^{i+j}$, and we fix a grading
\begin{equation}\label{grading_of_g}
    \underbrace{\mfk{g}_{-3} \oplus \mfk{g}_{-2} \oplus \mfk{g}_{-1}}_{\mfk{g}_-} \oplus \underbrace{\mfk{g}_{0} \oplus \mfk{g}_{1} \oplus \mfk{g}_{2} \oplus \mfk{g}_{3}}_{\mfk{p}},
\end{equation}
$[\mfk{g}_i,\mfk{g}_j] \subseteq \mfk{g}_{i+j}$, where $\mfk{g}^i = \mfk{g}_i \oplus \dots \oplus \mfk{g}_3$. Further, $\mfk{p}_+ \coloneqq \mfk{g}^1$, $\mfk{p} = \mfk{g}^0$, $\mfk{g}/\mfk{p} \cong \mfk{g}_-$. The filtration is $P$-invariant and the grading is $G_0$-invariant, where $\mfk{g}_0 = \on{Lie}(G_0)$.

\subsection{Regular Cartan connection}
In order to encode the geometric properties of a $(2,3,5)$-distribution in the form of a parabolic geometry $p\colon \mc{G} \to \mc{M}$ of type $(G_2,P)$, we need it defines the same filtration on $T\mc{M}$ (\cite[pp.\! 246-270]{parabook}). The filtration induced by $\mc{D}$ is
\begin{equation}\label{filtration_by_D}
    \begin{gathered}
        T\mc{M} = \mc{D}^3 \supset \mc{D}^2 \supset \mc{D}^1 \coloneqq \mc{D},\\
        \mc{D}^i = \mc{D}^{i-1} + [\mc{D}, \mc{D}^{i-1}] \quad \text{for } i\in \{2,3\}
    \end{gathered}
\end{equation}
and, using the composition
\begin{equation*}
    \mfk{g} \xlra{\omega^{-1}} T\mc{G} \xlra{p_*} T\mc{M},
\end{equation*}
the filtration induced by the Cartan connection $\omega$ is
\begin{equation}\label{filtration}
    \begin{gathered}
        T\mc{M} = T^{-3}\mc{M} \supset T^{-2}\mc{M} \supset T^{-1}\mc{M},\\
        p_* \bigl( \omega^{-1} (\mfk{g}^i) \bigr) = T^{i}\mc{M} \quad \text{for } i\in \{-3,-2,-1\}.
    \end{gathered}
\end{equation}

Now, suppose both filtrations \eqref{filtration_by_D}, \eqref{filtration} to be the same, $\mc{D}^i = T^{-i}\mc{M}$.
Both approaches induce an algebraic binary operation on the graded space $\on{gr} T\mc{M}$,
\begin{equation}\label{grTM}
    \on{gr} T\mc{M} = T^{-3}\mc{M}/T^{-2}\mc{M} \oplus T^{-2}\mc{M}/T^{-1}\mc{M} \oplus T^{-1}\mc{M},
\end{equation}
which \eqref{filtration_by_D} induces by Lie brackets of vector fields on $T\mc{M}$ (so called \emph{Levi brackets}) and \eqref{filtration} by Lie brackets of the Lie algebra $\mfk{g}_-$. We call the parabolic geometry \emph{regular}, if these two define the same Lie algebra structure on $\on{gr} T\mc{M}$, i.e.
\begin{equation}\label{gr_g-}
    \on{gr}T\mc{M} \cong \on{gr} \mfk{g}_- = \mfk{g}^{-3} / \mfk{g}^{-2} \oplus \mfk{g}^{-2} / \mfk{g}^{-1} \oplus \mfk{g}^{-1} / \mfk{g}^0.
\end{equation}
It fact, the parabolic geometry is regular (\cite[p.\! 256]{parabook}) if and only if its curvature $\kappa$ contains only terms of positive homogeneities.

\begin{remark}
    Consider the reduction $\mc{G}_0 \xlra{p_0} \mc{M}$ of the principal bundle $\mc{G} \xlra{p} \mc{M}$ to the subgroup $G_0 \subset P$, where $G_0$ is such that $\mfk{g}_0 = \on{Lie}(G_0)$. Since the $(2,3,5)$-geometries are given by the filtration \eqref{filtration_by_D} only, i.e., without any further reduction, the principal bundle $\mc{G}_0$ is the whole graded frame bundle $\on{Fr_{gr}}\mc{M}$ and $G_0 = \on{Aut_{gr}}(\mfk{g}_-)$.
\end{remark}

\begin{convention}
    From now, any considered Cartan connection will define the same filtration on $T\mc{M}$ and the same Lie algebra structure on $\on{gr} T\mc{M}$ as a given $(2,3,5)$-distribution $\mc{D}$. Therefore, any Cartan connection $\omega$ will be regular.
\end{convention}

\subsection{Normal Cartan connection}
There is a chain complex
$ \left( C_\bullet (\mfk{p}_+,\mfk{g}), \partial^* \right)$, with $C_k (\mfk{p}_+,\mfk{g}) = \bigwedge^k \mfk{p}_+ \otimes \mfk{g}$
and a cochain complex
$\left( C^\bullet (\mfk{g}_-,\mfk{g}), \partial \right)$, $C^k (\mfk{g}_-,\mfk{g}) = \bigwedge^k \mfk{g}_-^* \otimes \mfk{g}$.
The chains and cochains are linear maps between two graded spaces, hence they admit the filtration and grading by homogeneities as \eqref{linear_map_homogeneities}. The operators $\partial, \partial^*$ are the standard Lie algebra cohomology and homology differentials and they preserve the homogeneities. The operator $\partial$ is $G_0$-equivariant and $\partial^*$ is $P$-equivariant. For later use, let us write $\partial^* \colon C_2(\mfk{p}_+,\mfk{g}) \to C_1 (\mfk{p}_+, \mfk{g})$ explicitly (see \cite[p.~6]{dt2022}) for any basis element
\begin{equation}\label{del_star}
    \partial^*(X \wedge Y \otimes Z) = -Y \otimes [X,Z] + X \otimes [Y,Z] - [X,Y] \otimes Z
\end{equation}
and it is defined by linearity for any other element.

The Killing form on $\mfk{g}$ is nondegenerated and defines identifications
\begin{equation}\label{identities}
   (\mfk{g}/\mfk{p})^* \cong \,\mfk{p}_+ \cong (\mfk{g}_-)^*,
\end{equation}
where the first is an isomorphism of $P$-modules and the second is an isomorphism of $G_0$-modules. Moreover, one can define $\partial^*\colon C^{k+1} (\mfk{g}_-,\mfk{g}) \to C^{k} (\mfk{g}_-,\mfk{g})$ using an appropriate hermitian form and the identities \eqref{identities}. The differentials $\partial, \partial^*$ are adjoint to each other which leads to the \emph{Hodge decomposition}
\begin{equation}\label{hodge_decomposition}
    C^k(\mfk{g}_-, \mfk{g}) \cong \mathrlap{\,\overbrace{\phantom{\on{im} \partial \oplus \ker \square}}^{\ker \partial}} \on{im} \partial \oplus \underbrace{\ker \square \oplus \on{im} \partial^*}_{\ker \partial^*},
\end{equation}
where $\square = \partial \partial^* + \partial^* \partial$ is the \emph{Kostant Laplacian} and $\ker \square = \ker \partial \cap \ker \partial^*$.

Thanks to \eqref{identities}, we can consider $\kappa_u \in C^2 (\mfk{g}_-,\mfk{g})$ for the curvature function $\kappa$ in a point $u \in \mc{G}$ defined by \eqref{curvature_function}. We call the corresponding Cartan connection $\omega$ \emph{regular}, or \emph{normal}, if
\begin{equation} \label{Cc_regular_normal}
    \kappa_u \in C^2 (\mfk{g}_-, \mfk{g})^1,\ \text{or}\ \kappa_u \in {\ker \partial^*}\ \text{for any}\ u \in \mc{G},\ \text{respectively}.
\end{equation}

For a regular, normal geometry we define the \emph{harmonic curvature} $\kappa_H\colon \mc{G} \to H_2(\mfk{g}_-,\mfk{g})$
\begin{equation}
    \kappa_H = \kappa + \on{im} \partial^*,
\end{equation}
where
\begin{equation*}
    H^2(\mfk{g}_-,\mfk{g}) \cong \frac{\ker \partial}{\on{im} \partial} \cong \ker \square \cong \frac{\ker \partial^*}{\on{im} \partial^*} \cong H_2(\mfk{p}_+,\mfk{g}).
\end{equation*}

Since $\partial,\partial^*$ preserve homogeneities, the homology and cohomology groups $H_\bullet, H^\bullet$ inherit gradings. One can show (\cite[p.\! 7]{dt2022}), that $H^1(\mfk{g}_-, \mfk{g})^1 = \{0\}$ for $\mfk{g} = \on{Lie}(G_2)$ and our choice of $\mfk{p}$. It allows us to replace each parabolic geometry of type $(G_2,P)$ by an equivalent one which is normal and regular. Moreover, it is unique up to automorphisms of parabolic geometries.

\subsection{Normalization procedure} \label{section:normalization_procedure}

The filtration on $T\mc{M}$ induces the filtration on $T\mc{G}$
\begin{equation}\label{filtration_on_TG}
\begin{gathered}
    T\mc{G} = T^{-3}\mc{G} \supset T^{-2}\mc{G} \supset T^{-1}\mc{G} \supset T^0\mc{G},\\
    T^{-i}\mc{G} \coloneqq p^{-1}(\mc{D}^i),\ T^0\mc{G} \coloneqq V\mc{G},
\end{gathered}
\end{equation}
and it has to be preserved by any Cartan connection $\omega$, i.e.
\[ p^{-1}(\mc{D}^{-i}) = T^i\mc{G} = \omega^{-1}(\mfk{g}^i) \text{ for } i\leq 0. \]
Of course, $T^i\mc{G}$ cannot be fixed for $i > 0$, as it would depend on $\omega$.

Consider two Cartan connections $\omega, \widetilde{\omega}$. It follows from equivariance that
\begin{equation*} 
    \left(\widetilde\omega - \omega \right) \big|_{T^0\mc{G}} = 0.
\end{equation*}
Now, suppose $(\widetilde\omega - \omega) \left(T^i \mc{G}\right) \subset \mfk{g}^{i+j}$, it can be shown (\cite[p.\! 260]{parabook}) that 
\begin{equation} \label{kappa_difference}
    \widetilde\kappa - \kappa \colon \mc{G} \to C^2(\mfk{g}_-, \mfk{g})^j.
\end{equation}
Concretely, in any point $u \in \mc{G}$, one can consider $\phi \coloneqq (\widetilde\omega - \omega)_u$ as%
\footnote{One can define it as $(\widetilde\omega - \omega)_u \circ \widetilde\omega_u^{-1}$ or $(\widetilde\omega - \omega)_u \circ \omega_u^{-1}$. Their difference belongs to $C^1(\mfk{g}_-,\mfk{g})^{j+1}$, hence it hides in the $\psi$ term in the equation \eqref{kappa_difference_concretely}.}
$\phi \in C^1(\mfk{g}_-, \mfk{g})^j$ and
\begin{equation} \label{kappa_difference_concretely}
    \widetilde\kappa_u = \kappa_u + \partial \phi + \psi, \quad \psi \in C^2(\mfk{g}_-, \mfk{g})^{j+1},
\end{equation}
where $\partial$ is the cohomology differential. On the other hand, $\widetilde\omega = \omega + \phi$ is again a Cartan connection matching the filtered structure on $T\mc{M}$ for any $\phi\colon \mc{G} \to C^1(\mfk{g}_-,\mfk{g})^1$. Since $\kappa$ decomposes into parts in $\on{im} \partial$ and $\ker \partial^*$, see \eqref{hodge_decomposition}, we can kill the $\on{im} \partial$ part homogeneity by homogeneity until $\kappa$ belong to $\ker \partial^*$ and $\omega$ is normal.
Altogether, we can start with an arbitrary regular Cartan connection $\omega$ and obtain the regular normal one in a finitely many steps.

\subsection{Local description} \label{section:local_describtion}
We transfer this problem down to $T\mc{M}$, as we have more direct way to describe grading on $T\mc{M}$. The key ingredient is the $P$-equivariance: $\omega$ can be reconstructed from its values $\omega_{\sigma(x)}$ in points of a section $\sigma$ of $\mc{G}$, and the $P$-equivariance of $\ker \partial^*$ extends normality of $\omega\big|_{\on{im} \sigma}$ to the entire $\omega$.

We will proceed from the local point of view on an open subset $U \subset \mc{M}$ where the vector fields \eqref{frame} form a frame and the distribution $\mc{D}$ is $(2,3,5)$ on $U$. Following \cite[p.\! 77]{parabook},
a choice of a frame on $U \subset \mc{M}$ compatible with the $(2,3,5)$-filtration \eqref{filtration_by_D} corresponds to a section $\sigma\colon U \to \mc{G}$, $p \circ \sigma = \on{id}_U$, such that the pullback of the Cartan connection $\sigma^*\omega$ defines an isomorphism of $T_x\mc{M} \cong \mfk{g}_-$ in any point $x \in \mc{M}$.

Let us consider any local section $\sigma$ and any one-form $\alpha \in \Omega^1(\mc{M},\mfk{g})$ that induces a grading-preserving isomorphism $\alpha_x \colon T_x\mc{M} \to \mfk{g}_-$ for any $x \in \mc{M}$. Since $\sigma$ defines a local trivialization $\mc{G} \simeq U \times G_2$, we can construct a Cartan connection $\omega$ from $\sigma_* \alpha$ using the equivariance property, and then $\alpha = \sigma^* \omega$. For another section $\hat\sigma$, one can obtain the pullback $\hat{\sigma}^*\omega$ from a compatibility condition involving transition functions and the equivariance property on $\mc{G}$.

\begin{remark}
    Regardless of the neighborhood we choose, the resulting regular normal parabolic geometry can be extended to a maximal subset and must be locally automorphic to any other regular normal parabolic geometry on this subset.
\end{remark}

\begin{convention}
    For the purpose of this paper, we are interested only in regular points of $\mc{M}$ at which the distribution $\mc{D}$ is genuinely a $(2,3,5)$-distribution. We exclude singular points (e.g.\! $v_1 = v_3 = 0$) and choose a connected component, so that a parabolic geometry of type $(G_2,P)$ exists on this set, still denoted by $\mc{M}$ for simplicity. Then the frame \eqref{frame} forms a basis of $T_x\mc{M}$ at each point $x$. Moreover, we assume that $\mc{M}$ is diffeomorphic to an open, simply connected subset of $\R^5$, so that we do not have to worry about topological obstructions.
\end{convention}

In order to be able to decompose any $\sigma^* \omega$ by homogeneities, we need to fix grading on $T\mc{M}$ by a choice of representatives of $\on{gr} T\mc{M}$. Let us use the frame \eqref{frame},
\begin{equation}\label{grading_on_TM}
    T\mc{M} = \bigoplus_{i= -3}^{-1} (T\mc{M})_{i} \coloneqq \langle X_5, X_4 \rangle \oplus \langle X_3 \rangle \oplus \langle X_2, X_1 \rangle \cong \on{gr}T\mc{M},
\end{equation}
then the $\sigma^*\omega$ is decomposed by homogeneities
\begin{equation}\label{omega_homogeneities}
    \sigma^*\omega = \sum_{k=0}^6 (\sigma^*\omega)_k, \quad (\sigma^*\omega)_k \left( (T\mc{M})_i \right) \subseteq \mfk{g}_{i+k}, \quad i\in \{-3,-2,-1\},
\end{equation}
and the part $(\sigma^*\omega)_0$ is the same for any $\omega$, because it corresponds to the structure \eqref{gr_g-} defined by $\mc{D}$. Now, it is clear from \eqref{kappa_difference} and \eqref{kappa_difference_concretely}, that we can change homogeneity terms of $\sigma^*\omega$ step-by-step without breaking previous homogeneity terms of $\sigma^*\kappa = \kappa \circ \sigma$. Throughout this procedure, we suppose that the section $\sigma$ is fixed and the changes affects just $\omega$.

When we fix a faithful representation of $\mfk{g}$, the connection $\omega$ can be expressed as a matrix of one-forms, as well as $\sigma^*\omega$. Then the matrix of the curvature $(\sigma^*\Omega)_i^j$ can be computed as
\begin{equation}\label{curvature_computation}
    (\sigma^*\Omega)_i^j = \di (\sigma^*\omega)_i^j + (\sigma^*\omega)_i^k \wedge (\sigma^*\omega)_k^j.
\end{equation}
For this purpose, we represent $\mfk{g}$ in the following form (\cite[p.\! 3]{dt2022})
\[ \begin{pmatrix}
   2z_1 + z_2 & b_{10} & b_{11} & \sqrt{2} b_{21} & b_{31} & b_{32} & 0 \\
   a_{10} & z_1 + z_2 & b_{01} & \sqrt{2} b_{11} & - b_{21} & 0 & -b_{32} \\
   a_{11} & a_{01} & z_1 & -\sqrt{2} b_{10} & 0 & b_{21} & -b_{31} \\
   \sqrt{2} a_{21} & \sqrt{2} a_{11} & -\sqrt{2} a_{10} & 0 & \sqrt{2} b_{10} & -\sqrt{2} b_{11} & -\sqrt{2} b_{21} \\
   a_{31} & -a_{21} & 0 & \sqrt{2} a_{10} & -z_1 & -b_{01} & -b_{11} \\
   a_{32} & 0 & a_{21} & -\sqrt{2} a_{11} & -a_{01} & -z_1 - z_2 & -b_{10} \\
   0 & -a_{32} & -a_{31} & -\sqrt{2} a_{21} & -a_{11} & -a_{10} & -2z_1 - z_2
\end{pmatrix}, \]
where the coefficients
\[ a_{32}, a_{31}, a_{21}, a_{11}, a_{10}, a_{01}, z_2, z_1, b_{01}, b_{10}, b_{11}, b_{21}, b_{31}, b_{32} \]
corresponds to a basis of $\mfk{g}$ which we denote
\begin{equation}\label{basis_g}
    \underbrace{e_{-3,2}, e_{-3,1}}_{\mfk{g}_{-3}}, \underbrace{e_{-2,1}}_{\mfk{g}_{-2}}, \underbrace{e_{-1,2}, e_{-1,1}}_{\mfk{g}_{-1}}, \underbrace{e_{0,1}, e_{0,2}, e_{0,3}, e_{0,4}}_{\mfk{g}_{0}}, \underbrace{e_{1,1}, e_{1,2}}_{\mfk{g}_{1}}, \underbrace{e_{2,1}}_{\mfk{g}_{2}}, \underbrace{e_{3,1}, e_{3,2}}_{\mfk{g}_{3}},
\end{equation}
and the identification \eqref{identities} via the Killing form reads as (\cite[p.~6]{dt2022})
\begin{equation}\label{killing_identity}
    (e_{-3,2}^*, e_{-3,1}^*, e_{-2,1}^*, e_{-1,2}^*, e_{-1,1}^*) \mapsto (\tfrac{1}{8}e_{3,2}, \tfrac{1}{8}e_{3,1}, \tfrac{1}{24}e_{2,1}, \tfrac{1}{24}e_{1,2}, \tfrac{1}{24}e_{1,1}),
\end{equation}
where $e_{i,k}^*$ are covectors dual to $e_{i,k}$.

\subsection{Implementation}
Denote the frame \eqref{frame} in reverse order
\begin{equation} \label{pullback_frame}
    X_{-3,2}, X_{-3,1}, X_{-2,1}, X_{-1,2}, X_{-1,1}
\end{equation}
and its dual frame
\begin{equation} \label{pullback_coframe}
    \theta^{-3,2}, \theta^{-3,1}, \theta^{-2,1}, \theta^{-1,2}, \theta^{-1,1},
\end{equation}
so that it corresponds to the chosen grading \eqref{grading_on_TM}. Now, we can decompose $\sigma^*\omega$ by homogeneity terms
\begin{equation} \label{sigma_omega_terms}
    \sigma_x^*\omega = \sum_{h=0}^6 (\sigma_x^*\omega)_h = \sum_{\substack{i,j,h \\ i<0}} \left( \sum_k a_{i,k}^{h,j} (x)\ \theta^{i,k} \right) e_{i+h,j}.
\end{equation}
The pullback of curvature $\Omega$ decomposes by homogeneity terms into
\begin{equation}
    \sigma^*\Omega = \sum_{h=-3}^9 (\sigma^*\Omega)_h = \sum_{\substack{h,k_3 \\ i_1,i_2<0}} \left( \sum_{k_1,k_2} \Omega_{I,K}^{h,j} (x)\ \theta^{i_1,k_1} \wedge \theta^{i_2,k_2} \right) e_{i_1+i_2+h,k_3},
\end{equation}
and the curvature function $\kappa\colon \mc{G} \to C^2(\mfk{g}_-, \mfk{g})$ into
\begin{equation}
    \kappa = \sum_{h=-3}^9 \kappa_h = \sum_{\substack{h,k_3 \\ i_1,i_2<0}} \left( \sum_{k_1,k_2} \kappa_{I,K}^{h,j} (x)\ e^*_{i_1,k_1} \wedge e^*_{i_2,k_2} \right) e_{i_1+i_2+h,k_3},
\end{equation}
where the multi-indices $I=(i_1,i_2)$, $K=(k_1,k_2)$.

By factoring out the one-forms $\theta^{i,k}$ in \eqref{sigma_omega_terms} and forgetting the $e_{i',k'}$-terms for $i'\geq0$, we find the image of the vector fields $X_{i,j}$ under the map $\mathrm{pr_{\mfk{g}/\mfk{p}}} \circ\sigma_x^*\omega$. We use the inverse of this map to find the curvature function $\sigma^*\kappa$.

\begin{algorithm} \label{normalization_algorithm}
    Suppose we have a frame \eqref{pullback_frame} and its dual \eqref{pullback_coframe} of $T\mc{M}$ adapted to the $(2,3,5)$ filtration of $T\mc{M}$. By the following steps, we obtain a pullback $\sigma^*\omega$ of a regular normal Cartan connection $\omega$ for the filtered structure defined by $\mc{D}$.

    \begin{algorithmic}[1]
        \algrenewcommand\algorithmicrequire{\textbf{Input:}}
        \algrenewcommand\algorithmicensure{\textbf{Output:}}
        \Require $\di\theta^{i,k}$, $\di A_a^k$\dots their expressions \eqref{coefficient_function_differential}, \eqref{structure_equation}
        \Ensure $\sigma^*\omega$ \dots pullback of the regular normal Cartan connection
        \State find $a_{i,k}\in \R$ such that $\sigma^*\omega = \sum_{\substack{i,j \\ i<0}} \left(a_{i,k}\ \theta^{i,k} \right) e_{i,k}$ satisfies $\sigma^*\kappa \in C(\mfk{g}_-,\mfk{g})^1$
        \For{$h$ from $1$ to $6$}
            \State denote $t_h = \sum_{\substack{i,j \\ i<0}} \left( \sum_k a_{i,k}^{h,j} (x)\ \theta^{i,k} \right) e_{i+h,j}$
            \State set $\sigma^*\omega = \sigma^*\omega + t_h$
            \State compute $\sigma^* \Omega$
            \State find the inversion $\sigma_x^*\omega^{-1}$ of $\mathrm{pr_{\mfk{g}/\mfk{p}}} \circ\sigma_x^*\omega$
            \State compute preimages $\widetilde{X}_{i,k} = \sigma^*\omega^{-1} \left(e_{i,k}\right)$ for $i<0$
            \State \parbox[t]{\linewidth-\algorithmicindent}{extract $e_{h+i_1+i_2,k_3}$-terms from $\sigma^*\Omega \left(\widetilde{X}_{i_1,k_1}, \widetilde{X}_{i_2,k_2}\right)$ and define $\sigma_x^*\kappa_h$ as their sum}
            \State using \eqref{killing_identity}, rewrite $\kappa_h \in C_2(\mfk{g}_-,\mfk{g})$ into $\kappa_h \in C^2(\mfk{p}_+,\mfk{g})$
            \State solve $\partial^*\left(\sigma^*\kappa_h\right) = 0$ defined by \eqref{del_star} for the unknowns $a_{i,k}^{h,j} (x)$
            \State substitute the solution into $\sigma^*\omega$
        \EndFor
        \State set the rest of unknowns in $\sigma^*\omega$ equal to $0$
        \State \Return $\sigma^*\omega$
    \end{algorithmic}
\end{algorithm}
\begin{proof}
    In the first step, we obtain $\sigma^* \omega$ coming from a regular Cartan connection $\omega$, since $\sigma^* \omega$ in the considered form defines an isomorphism $T_x\mc{M} \cong \mfk{g}/\mfk{p}$ and $\sigma^*\kappa \in C(\mfk{g}_-,\mfk{g})^1$. The requirement $\sigma^*\kappa \in C(\mfk{g}_-,\mfk{g})^1$ is, by \eqref{curvature_function}, equivalent to
    \[ \sigma^*\omega ([X_{i,k_1}, X_{j,k_2}]) \equiv [\sigma^*\omega (X_{i,k_1}), \sigma^*\omega(X_{j,k_2})] \mod \mfk{g}^{i+j+1} \]
    and this equality holds, since the vectors $e_{i,k}$ admit the same relations as $X_{i,k}$ up to a constant coefficient.

    The terms of $\kappa_h$ are precisely $e_{h+i_1+i_2,k_3}$-terms of $\kappa(e_{i,k_1}, e_{j,k_2})$ and the way, how we obtain the corresponding pullback is justified by
    \[ \sigma_x^*\kappa = (\kappa \circ \sigma)(x)(e_{i,k_1}, e_{j,k_2}) = \Omega(\omega_{\sigma(x)}^{-1}(e_{i,k_1}), \omega_{\sigma(x)}^{-1}(e_{j,k_2})) \]
    and
    \[ \sigma_x^*\omega \left(\widetilde{X}_{i,k}\right) = \omega_{\sigma(x)}\left( \sigma_* \widetilde{X}_{i,k}\right) = e_{i,k} + p, \quad p \in\mfk{p}.\]
    
    Following the normalization procedure, see subsection \ref{section:normalization_procedure} above, all other steps would be easily seen as correct if we worked directly with $\omega$ instead of $\sigma^*\omega$. The way to reconstruct $\omega$ from $\sigma^*\omega$ is explained in \ref{section:normalization_procedure}, too. The existence of the coefficients $a_{i,k}^{h,j}$ in each step $h$ follows from the existence of a regular normal Cartan connection differing from $\omega$ by terms of homogeneity $> h$, and its pullback. The claim about the relation between $\widetilde\omega-\omega$ and $\widetilde\kappa-\kappa$, see \eqref{kappa_difference} and \eqref{kappa_difference_concretely}, holds true also for their pullbacks,
    \[ \sigma^*\widetilde\omega - \sigma^* \omega \in \Omega^1 (\mc{M},\mfk{g})_h
    \Rightarrow (\widetilde\omega - \omega)\big|_{\on{im} \sigma} \in \Omega^1 (\on{im} \sigma, \mfk{g})^h
    \Rightarrow \sigma^* (\widetilde\kappa - \kappa) \in C^2 (\mfk{g}_-, \mfk{g})^h. \]
    Using the obvious fact that $\ker \partial^*$ is preserved by equivariancy, we deduce that $\partial^*\left(\sigma^*\kappa_h\right) = 0$ implies $\partial^*\kappa_h = 0$. Since the $h$-iteration of the \textbf{\em for} cycle affects only the $C^1(\mfk{g}_-, \mfk{g})^{h+1}$-part of $\kappa$, after $6$ steps we obtain $\sigma^*\omega$ such that the corresponding $\kappa$ satisfies $\sigma^*\kappa \in C^2(\mfk{g}_-,\mfk{g})^1$ and $\partial^*\kappa_k = 0$ for $k\in\{-3, -2, \dots, 6\}$. For $h \in \{7,8,9\}$, $\partial^*\kappa_h=0$ always, because $\partial^*\colon C_2(\mfk{p}_+,\mfk{g}) \to C_1(\mfk{p}_+,\mfk{g})$ preserves homogeneities and $C_1(\mfk{p}_+,\mfk{g})$ has no element of homogeneity $>6$.
\end{proof}

Let us consider any branch of $(\ell_1,\ell_3,b)$ from Figure \ref{fig:branches}. In this branch, all Lie algebras $\mfk{s}_{\ell_1,\ell_3,b}$ share the same basis vector fields, therefore the multiplication table and the expressions for $\di \theta^{i,j}, \di A_a^k$ from \eqref{coefficient_function_differential}, \eqref{structure_equation} can be written as single $(\ell_1,\ell_3,b)$-depended expressions. The expressions for $\di \theta^{i,j}, \di A_a^k$ with $A_a^k$ as variables are significantly simpler than the expressions for $\di \theta^{i,j}$ in coordinates, and we need them to compute $\sigma^* \Omega$ in the Algorithm \ref{normalization_algorithm}. It allows us to compute the normal Cartan connection with $A_a^k$ as variables and $\ell_1,\ell_3,b$ as parameters at once.

\begin{example}
    For illustration, we show the part of the algorithm for homogeneity one. We start with $\sigma^*\omega$ in the following form
    \begin{equation*}
    \begin{aligned}
        \sigma^*&\omega =  6\, \theta^{-3,2}\ e_{-3,2} - 6\, \theta^{-3,1}\ e_{-3,1} - 2\, \theta^{-2,1}\ e_{-2,1} + \theta^{-1,2}\ e_{-1,2} + \theta^{-1,1}\ e_{-1,1}\\
        &=
        \begin{pmatrix}
            0 & 0 & 0 & 0 & 0 & 0 & 0 \\
            \theta^{-1,1} & 0 & 0 & 0 & 0 & 0 & 0 \\
            \theta^{-1,2} & 0 & 0 & 0 & 0 & 0 & 0 \\
            -2\sqrt{2}\,\theta^{-2,1} & \sqrt{2}\, \theta^{-1,2} & -\sqrt{2}\, \theta^{-1,1}    & 0 & 0 & 0 & 0 \\
            -6\, \theta^{-3,1} & 2\, \theta^{-2,1} & 0 & \sqrt{2}\, \theta^{-1,1} & 0 & 0 &     0 \\
            6\, \theta^{-3,2} & 0 & -2\, \theta^{-2,1} & -\sqrt{2}\, \theta^{-1,2} & 0 & 0 &    0 \\
            0 & -6\, \theta^{-3,2} & 6\, \theta^{-3,1} & 2\sqrt{2}\, \theta^{-2,1} & -\theta^{-1,2} & -\theta^{-1,1} & 0
        \end{pmatrix}.
    \end{aligned}
    \end{equation*}
    We can simply compute by \eqref{curvature_computation} that $\sigma^*\Omega \in \Omega^2(\mc{M},\mfk{g})^1$ and $\on{im} \sigma^*\kappa \subseteq C^2(\mfk{g}_-,\mfk{g})^1$ as well, i.e.~the corresponding $\omega$ is regular. We add terms of homogeneity one, $\sigma^*\widetilde\omega \coloneqq \sigma^*\omega + t_1$,
    \[t_1 = \sum_{\substack{i,j \\ i<0}} \left( \sum_k a_{i,k}^{1,j} (x)\ \theta^{i,k} \right) e_{i+1,j},\]
    where the functions $a_{i,k}^{h,j} (x)$ are unknown. Now, $\sigma^*\widetilde\omega$ is of the form
    \begin{equation*}
    \begin{aligned}
        \sigma^* \widetilde\omega &=  6\, \theta^{-3,2}\ e_{-3,2} - 6\, \theta^{-3,1}\ e_{-3,1} + \left(-2\, \theta^{-2,1} + \sum_{i=1}^2 a_{-3,i}^{1,1} (x)\, \theta^{-3,i}\right)\ e_{-2,1}\\
        &\qquad + \left( \theta^{-1,2} + a_{-2,1}^{1,2} (x)\, \theta^{-2,1}\right)\ e_{-1,2} + \left( \theta^{-1,1} + a_{-2,1}^{1,1} (x)\, \theta^{-2,1}\right)\ e_{-1,1}\\
        &\qquad + \sum_{j=1}^4 \left( \sum_{i=1}^2 a_{-1,i}^{1,j} (x)\, \theta^{-1,i}\right) e_{0,j}.
    \end{aligned}
    \end{equation*}
    Thus, the first column in the matrix form is
    \begin{equation*}
    \begin{pmatrix}
        (2 a_{-1,1}^{1,2}(x) + a_{-1,1}^{1,3}(x))\ \theta^{-1,1} + (2 a_{-1,2}^{1,2}(x) + a_{-1,2}^{1,3}(x))\ \theta^{-1,2} & \cdots \\
        \theta^{-1,1} + a_{-2,1}^{1,1}(x)\ \theta^{-2,1} & \cdots \\
        \theta^{-1,2} + a_{-2,1}^{1,2}(x)\ \theta^{-2,1} & \cdots \\
        \sqrt{2}(-2\,\theta^{-2,1} + a_{-3,1}^{1,1}(x)\ \theta^{-3,1} + a_{-3,2}^{1,1}(x)\ \theta^{-3,2}) & \cdots\\
        -6\, \theta^{-3,1} & \cdots\\
        6\, \theta^{-3,2} & \cdots \\
        0 & \cdots
    \end{pmatrix}.
    \end{equation*}
    The $\mfk{g}_-$-part of $\sigma^*\omega$ is
    \begin{equation*}
    \begin{aligned}
        \mathrm{pr}_{\mfk{g}_-} \circ \sigma^*\widetilde\omega &= 
        6\, \theta^{-3,2} e_{-3,2}  -6\, \theta^{-3,1} e_{-3,1}\\
        &\qquad {} + \tfrac{1}{\sqrt{2}} \big({-2}\sqrt{2}\,\theta^{-2,1} + a_{-3,1}^{1,1}(x)\ \theta^{-3,1} + a_{-3,2}^{1,1}(x)\ \theta^{-3,2} \big) e_{-2,1}\\ 
        &\qquad {} + \big(\theta^{-1,2} + a_{-2,1}^{1,2}(x)\ \theta^{-2,1} \big) e_{-1,2} + \big( \theta^{-1,1} + a_{-2,1}^{1,1}(x)\ \theta^{-2,1} \big) e_{-1,1}.
    \end{aligned}
    \end{equation*}
    and we use its inverse to find vector fields
    \begin{equation*}
        (\widetilde{X}_{-3,2},\widetilde{X}_{-3,1},\widetilde{X}_{-2,1},\widetilde{X}_{-1,2},\widetilde{X}_{-1,1}),
    \end{equation*}
    such that
    \[ \sigma ^* \widetilde\omega (\widetilde{X}_{-i,j}) \equiv e_{-i,j} \mod \mfk{p}. \]
    
    Next, we compute $\sigma^*\widetilde\Omega$ using \eqref{curvature_computation} and express terms of homogeneity one of $\sigma^*\widetilde\kappa = \sigma^* \left(\widetilde\Omega \circ (\widetilde\omega^{-1},\widetilde\omega^{-1})\right)$. Define $\widetilde\Theta^{i,j} \in \Omega^2(\mc{M})$ by
    \begin{equation}\label{theta_coeffs_eij}
        \sigma^*\widetilde\Omega = \sum_{i,j} \widetilde\Theta^{i,j}\, e_{i,j},
    \end{equation}
    then terms of homogeneity one are
    \begin{equation*}
        (\sigma^*\widetilde\kappa)_1 = \sum_{\substack{i < j < 0\\ k_1, k_2, k_3}} \widetilde\Theta^{1+i+j,k_3}(\widetilde{X}_{i,k_1}, \widetilde{X}_{j,k_2})\ \left( e_{i,k_1}^* \wedge e_{j,k_2}^* \otimes e_{1+i+j,k_3} \right) \in C^2(\mfk{g}_-,\mfk{g})_1.
    \end{equation*}
    We change this term $(\sigma^*\widetilde\kappa)_1 \in C^2(\mfk{g}_-,\mfk{g})_1$ to $C_2(\mfk{p}_+,\mfk{g})^1$ using \eqref{killing_identity} and compute $\partial^* (\sigma^*\widetilde\kappa)_1$ using \eqref{del_star} for each term $e_{-i,k_1} \wedge e_{-j,k_2} \otimes e_{1+i+j,k_3}$ separately. Then we find the functions $a_{i,k}^{h,j} (x)$ by solving the equation
    \begin{equation*}
        \partial^* (\sigma^*\widetilde\kappa)_1 = 0.
    \end{equation*}

    We continue by adding terms of successive homogeneities one by one and requiring $\partial^* (\sigma^*\widetilde\kappa)_h = 0$ for each homogeneity $h$ up to and including homogeneity $6$. Then we finally obtain the pullback of the regular normal Cartan connection.
    \hfill$\triangle$
\end{example}

\section{The harmonic curvature} \label{section:harmonic_curvature}

Suppose we have computed a regular normal Cartan connection $\omega_N$ and its curvature $\Omega_N$. Similarly to \eqref{theta_coeffs_eij}, denote $\Omega_N^{i,k}$ the coefficients of $e_{i,k}$ in $\sigma^*\Omega_N$, and $X^N_{i,k}$ the frame that is mapped by $\sigma^*\omega_N$ to $[e_{i,k}] \in \mfk{g}/\mfk{p}$.

According to \cite[pp. 6, 25-26]{dt2022}, the harmonic curvature $\kappa_H \in H_2(\mfk{p}_+,\mfk{g}) \cong S^2(\mfk{g}_{-1})^*$, and it corresponds to the binary quartic form
\begin{equation} \label{cartan_biquartic}
    Q(x_1,x_2) = A_1 {x_2}^4 + 4A_2 x_1 {x_2}^3 + 6 A_3 {x_1}^2 {x_2}^2 + 4 A_4 {x_1}^3 x_2 + A_5 {x_1}^4
\end{equation}
with the coefficients $A_1, \dots, A_5$ computed by
\begin{equation*}
\begin{gathered}
    A_1 = -\Omega_N^{0,1} (X^N_{-3,1}, X^N_{-1,1}),\quad A_2 = \Omega_N^{0,3} (X^N_{-3,1}, X^N_{-1,1}),\quad A_3 = \Omega_N^{0,1} (X^N_{-3,2}, X^N_{-1,2})\\
    A_4 = -\Omega_N^{0,3} (X^N_{-3,2}, X^N_{-1,2}),\quad A_5 = -\Omega_N^{0,4} (X^N_{-3,2}, X^N_{-1,2}).
\end{gathered}
\end{equation*}

Investigating the root type of the binary quartic $Q$, we can estimate the number of symmetries of $(\mc{M},\mc{D}_{\ell_1,\ell_3,b})$. The root type of $Q$ is determined by vanishing of $Q$, the covariant $H$ or the invariant $\Delta$, and we will investigate it for each branch of $(\ell_1,\ell_3,\ell_3)$ separately.

\begin{proposition} \label{prop:biquartic}
    Consider a regular normal Cartan geometry of type $(G_2,P)$ over a filtered manifold $(\mc{M,\mc{D}})$ with the Cartan connection $\omega_N$ and the binary quartic $Q$ defined by \eqref{cartan_biquartic} with a constant root-type. Denote $\mc{S}$ the Lie algebra of infinitesimal symmetries of $(\mc{M},\mc{D})$. The upper bound for the dimension of $\mc{S}$ is determined by the binary quartic $Q$ as follows.
    \begin{enumerate}[label=\arabic*)]
        \item $Q \equiv 0$, if and only if $\dim \mc{S} = 14$;
        \item if $Q \not\equiv 0$ and the hessian $H$ of $Q$ vanishes,
        then $\dim (\mc{S}) \leq 7$, where
            \begin{equation*}
            \begin{aligned}
                H(x_1,x_2) &\coloneqq (A_3 A_5 - {A_4}^2) {x_1}^4 + 2 (A_2 A_5 - A_3 A_4) {x_1}^2 x_2\\
                &\phantom{\coloneqq= {}} + (A_1 A_5 + 2 A_2 A_4 - 3 {A_3}^2) {x_1}^2 {x_2}^2\\
                &\phantom{\coloneqq= {}} + 2 (A_1 A_4 - A_2 A_3) {x_1} {x_2}^3 + (A_1 A_3 - {A_2}^2) {x_2}^4;
            \end{aligned}
            \end{equation*}
        \item if $H \not\equiv 0$ and the discriminant $\Delta$ of $Q$ vanishes, then $\dim (\mc{S}) \leq 6$, where
            \begin{equation*}
            \begin{gathered}
                \Delta \coloneqq i^3 - 27j^2,\\
                i\coloneqq A_1 A_5 - 4 A_2 A_4 + 3 {A_2}^2,\\
                j \coloneqq A_1 A_3 A_5 + 2A_2 A_3 A_4 - {A_3}^3 - {A_2}^2 A_5 - A_1 {A_4}^2;
            \end{gathered}
            \end{equation*}
        \item if $\Delta \neq 0$, then $\dim (\mc{S}) \leq 5$.
    \end{enumerate}
\end{proposition}
\begin{proof}
    The correspondence between the root type of $Q$ and the upper bound of $\dim \mc{S}$ is described in \cite{bkdt2017}[Proposition 5.2.1]. For the correspondence between the root types of $Q$ and invariants and covariants, see \cite{olver}[p.~29].
\end{proof}

\section{Conclusions and further remarks} \label{section:results}

Our intention was to determine whether there are parameters $\ell_1,\ell_3,b$ resulting in a locally flat model of Cartan geometry of type $(G_2,P)$, and ideally, to find parameters for the robot with $\dim \mc{S} \geq 6$. We proceeded in the following steps for each branch separately. We computed:
\begin{enumerate}[label=\arabic*)]
    \item the structure coefficients $C_{\alpha\beta}^\gamma$, see \eqref{structure_coefficients},
    \item the expression of the functions $A_a^k$ in coordinates, see \eqref{Aak_functions},
    \item the expression of the exterior derivative $\di \theta^i$ and $\di A_a^k$ using \eqref{structure_equation}, \eqref{coefficient_function_differential} and the simplification \eqref{Aai_relations}, but we do not substitute the coordinate expression in $A_a^k$,
    \item the parameter-depended normal Cartan connection $\omega$ and its curvature function $\kappa$ by Algorithm \ref{normalization_algorithm},
    \item the Cartan binary quartic $Q$ with $A_a^k$ as variables and $\ell_1, \ell_3, b$ as parameters,
    \item the coordinate expression of $Q$, i.e.~we substitute into $A_a^k$,
    \item parameters for which $Q \equiv 0$.
\end{enumerate}

We have obtained the regular normal Cartan connection and its curvature function for each branch in a compact form with $A_a^k$ as variables and $\ell_1, \ell_3, b$ as parameters. However, the coefficient functions $A_a^k$ are not independent. Using the coordinates leading to the polynomial form \eqref{polynomial_generators} of $X_i$, we find all $A_a^k$ to be rational functions in $2$ of the coordinates. Of course, we can also find several cubic polynomial relations from \eqref{coefficient_function_differential} by $\di (\di A_a^k) = 0$. Therefore, to investigate root-types of $Q$, we needed to use the coordinate expression, i.e.~to substitute into $A_a^k$, which made the coefficients of binary quartic rational functions with numerators of large degree, and our limited resources made it computationally impossible to check vanishing of the invariants from Proposition \ref{prop:biquartic}, except $Q \equiv 0$.

\begin{theorem} \label{thm:results}
    There are no parameters $\ell_1, \ell_3, b$ such that the filtered structure $(\mc{M}, \mc{D}_{\ell_1, \ell_3, b})$ of the robot admits $14$-dimensional algebra of infinitesimal symmetries.
\end{theorem}
\begin{proof}
    For each branch from Proposition \ref{prop:branches}, we compute the structure coefficients $C_{\alpha \beta}^\gamma$ of the Lie algebra $\mfk{s}_{\ell_1, \ell_3, b}$ and the coefficient functions $A_a^k$ expressed in the polynomial coordinates in the Maple worksheet \cite[\texttt{vectorfields.mw}]{attached_files} and save them into \cite[\texttt{branches.m}]{attached_files}.
    
    The normalization Algorithm \ref{normalization_algorithm} using the relations \eqref{Aai_relations} for $A_a^k$ is realized for each branch in \cite[\texttt{normalization.mw}]{attached_files}. To save the computational time, we perform the normalization algorithm only up to homogeneity $4$. It is sufficient for finding the binary quartic, since it contains only coefficients of $\kappa_4$ that does not change in other steps of the algorithm. The resulting binary quartics are saved in \cite[\texttt{biquartics.m}]{attached_files}. In \cite[flatcase.mw]{attached_files}, we substitute the polynomial coordinate expressions into $A_a^k$, and search for parameters, where the binary quartic $Q$ (and therefore also $\kappa$) vanish. We find that it does not happen for any parameter.
\end{proof}

The question of whether there is at least one distribution $\mc{D}_{\ell_1, \ell_3, b}$ whose symmetry group acts locally transitively remains unanswered. To use our resulting binary quartics to get an answer, we would require either a more suitable substitution into $A_a^k$ or greater computational power. Further, one could exploit the resulting normal Cartan connection for finding the Lie algebra of infinitesimal symmetries $\mc{S}$ for concrete parameters, in particular for the standard snake.

The technique presented in this paper could likewise be applied to other robots with bracket-generating distributions or filtered manifolds that depend smoothly on parameters. Although, the requirement on a distribution to contain basis vector fields generating a finite dimensional Lie algebra is very restrictive. Also the role of the notion of the growth vector with respect to $\mathfrak{d}$ (already defined in \cite{md2024}) remains unclear to us.

Instead of searching for homogeneous spaces of $(2,3,5)$-geometries between the branches, we could also ask in a reverse way, whether we can classify all finite dimensional Lie algebras occurring for homogeneous spaces of $(2,3,5)$-geometries. We have no doubt that the theory of Section \ref{section:coframe_structure} could be developed further or specified for particular geometries. This setting can also be viewed as a special case of finite dimensional Lie subalgebras of $\mfk{X}(\R^n)$, a widely studied but in general complicated problem.

\nocite{*}

\end{document}